\documentclass[12pt,times]{article}

\usepackage{graphicx}
\usepackage{epsfig}
\usepackage{amssymb}
\usepackage[left,pagewise,mathlines]{lineno}

\newtheorem{theorem}{Theorem}

\newtheorem{lemma}[theorem]{Lemma}
\newtheorem{corollary}[theorem]{Corollary}

\newenvironment{proof}{
\par
\noindent {\bf Proof.}\rm}%
{\mbox{}\hfill\rule{0.5em}{0.809em}\par}

\title{\bf \normalsize
\Large The Target Set Selection Problem on
 Cycle Permutation Graphs, Generalized Petersen Graphs
 and
 Torus Cordalis}

\author{\small Chun-Ying~Chiang\thanks{Partially supported by
 National Science Council under grant NSC97-2628-M-008-018-MY3.},
 Liang-Hao~Huang\thanks{Partially supported by
 National Science Council under grant NSC98-2811-M-008-072.},
 Wei-Ting~Huang,
 Hong-Gwa
Yeh\thanks{Partially supported by National
 Science Council under
grant NSC100-2811-M-008-052}
 \thanks{Corresponding author (hgyeh@math.ncu.edu.tw)}\\
{\footnotesize \em Department of Mathematics, National Central
University, Taiwan}}

\date{}

\begin{document}

\maketitle
\baselineskip=17pt

\begin{abstract}
  In this paper we consider a
  fundamental problem in the area of viral marketing, called
  T{\scriptsize ARGET} S{\scriptsize ET} S{\scriptsize ELECTION}
  problem.
  In a a viral marketing setting,
  social networks are modeled by graphs
  with potential customers of a new product as vertices and friend relationships as edges, where
   each vertex $v$ is assigned a threshold value
  $\theta(v)$. The thresholds represent the different latent tendencies of
  customers (vertices) to buy the new product when their friend
  (neighbors) do.
  Consider a repetitive process on social network
  $(G,\theta)$ where each vertex $v$ is associated with two states, active and
  inactive, which indicate whether $v$ is persuaded into buying the
  new product.
  Suppose we are given a target set $S\subseteq V(G)$.
  Initially, all vertices in $G$ are inactive.
  At time step $0$, we choose all vertices in $S$ to become active.
  Then, at
  every time step $t>0$, all vertices that were active in time step $t-1$
  remain active, and we activate any vertex $v$
  if at least $\theta(v)$ of
  its neighbors were active at time step $t-1$. The activation
  process terminates when no more vertices can get activated.
  We are interested in the following optimization problem,
  called
  T{\scriptsize ARGET}
  S{\scriptsize ET}
  S{\scriptsize ELECTION}:
    Finding a target set $S$ of smallest possible size
  that activates all vertices of $G$.
   There is an important and well-studied threshold called
   strict majority threshold, where
   for every vertex $v$ in $G$ we have
    $\theta(v)=\lceil{(d(v) +1)/2}\rceil$ and  $d(v)$ is the
    degree of  $v$ in $G$.
   In this paper,
   we consider the
   T{\scriptsize ARGET}
   S{\scriptsize ET}
   S{\scriptsize ELECTION} problem under
   strict majority thresholds
   and
   focus on three popular regular network structures:
   cycle permutation graphs,
   generalized Petersen graphs
   and
   torus cordalis.
\medskip

\end{abstract}
\parskip 4pt

\section{Introduction and preliminary results}  \label{intro}
 A {\em graph} $G$ consists of a set $V(G)$ of {\em vertices}
 together with a set $E(G)$ of unordered pairs of vertices called {\em edges}.
 We often use $uv$ for an edge $\{u,v\}$.
 Two vertices $u$ and $v$ are {\em adjacent}
 to each other if $uv \in E(G)$.

 %
 %
 In a viral marketing setting,
 a social network $(G,\theta)$ is a connected graph $G$ equipped with thresholds $\theta : V(G)\rightarrow \mathbb{Z}$,
 where
 each vertex represents a potential customer of a new product,
 and each edge indicates that the two people are friends.
 The thresholds represent the different latent tendencies of
 vertices (customers) to buy the new product
 when their neighbors (friends) do.
    There are three types of important and well-studied thresholds called
    {\em $k$-constant threshold},
   {\em majority threshold} and
   {\em strict majority threshold}.
   In a $k$-constant threshold,
   we have $\theta(v)=k$ for all vertices $v$ of $G$, and
   $(G,\theta)$ is abbreviated to $(G,k)$.
   In a majority threshold for every vertex $v$ in $G$ we have
   $\theta(v)=\lceil{d(v)/2}\rceil$, while in a strict majority
   threshold we have $\theta(v)=\lceil{(d(v) +1)/2}\rceil$,
   where $d(v)$ is the degree of $v$ in $G$.

 In a social network $(G,\theta)$,
 every vertex in $G$ has its own color which is either black or white, where black vertices represent {\em active} vertices, and
 white vertices represent {\em inactive} vertices.
 Given a set $S\subseteq V(G)$, consider the
 following repetitive 
 process on $(G,\theta)$
 called {\em activation process in  $(G,\theta)$ starting at
 target set $S$}.
 Initially (at time $0$),
 set all vertices in $S$ to be black
 (with all other vertices white).
  After that, at each time step,
  the states of vertices
  are updated according to the following rule:

  {\bf Parallel updating rule}:
   All inactive vertices $v$ that have at least
   $\theta(v)$ already-active neighbors become active.

  The activation process terminates
  when no more vertices can get activated.
  The set of vertices that are active at the end
  of the process is denoted by $[S]^G_\theta$.
  If $F\subseteq [S]^G_\theta$,
  then we say that the target set
  $S$ {\em influences} $F$ in $(G,\theta)$.
    We are interested in the following optimization problem,
  called
  T{\scriptsize ARGET}
  S{\scriptsize ET}
  S{\scriptsize ELECTION}:
    Finding a target set $S$ of smallest possible size
    that influences all vertices in the
    social network $(G,\theta)$,
    that is $[S]^G_\theta=V(G)$ (such set $S$ is called a
    {\em minimum seed} or an {\em optimal target set} for $(G,\theta)$).
  We define
  $\mbox{min-seed}(G,\theta)
  =\min\{|S|: S\subseteq V(G)$ and $[S]^G_\theta=V(G)\}$.

  %
  %
  %
  %
  The T{\scriptsize ARGET}
  S{\scriptsize ET}
  S{\scriptsize ELECTION} problem and some of its variants
  were introduced and studied in
  \cite{chen,Domingos2001,Dreyer+Roberts,
  kkt2003,kkt2005,Peleg1998,Peleg2002,
  Richardson2002,Roberts2003,Roberts2006}.
  It is not surprising that
  T{\scriptsize ARGET}
  S{\scriptsize ET}
  S{\scriptsize ELECTION} is NP-complete in general.
  Peleg \cite{Peleg2002} proved that it is NP-hard to
  compute the optimal target set for majority thresholds.
  In $k$-constant threshold setting,
  Dreyer and Roberts \cite{Dreyer+Roberts} showed that
  it is NP-hard to compute the min-seed$(G,k)$ for any
  $k\geq 3$, and Chen \cite{chen} showed that it is also NP-hard to
  compute min-seed$(G,2)$.
   More surprising is the fact that min-seed$(G,\theta)$
   is extremely hard to approximate.
   For any graph $G$ with majority thresholds $\theta$,
   Chen \cite{chen} proved that
   min-seed$(G,\theta)$
   cannot be
  approximated within the ratio
  $O(2^{\log^{1-\epsilon}n})$ for any fixed constant $\epsilon>0$,
  unless $NP\subseteq DTIME(n^{polylog(n)})$, where $n=|V(G)|$.

  Very little is known about the exact value of min-seed$(G,\theta)$.
  Related results can be found in
  \cite{Ackerman2010,Ben-zwi2011,Berger2001,chen,Yeh2011,
  Dreyer+Roberts,Flocchini2001,Flocchini2003,Flocchini2004,Flocchini2009,Luccio1998,
  Luccio1999,Peleg1998,Peleg2002}, where min-seed$(G,\theta)$ has
  been investigated under different threshold models for
  different types of network structure $G$:
  bounded treewidth graphs, trees,
  block-cactus graphs,  chordal graphs, Hamming graphs, chordal rings,
  tori, meshes, butterflies.
  %
  %
  %
   In the current paper,
   we consider
   T{\scriptsize ARGET}
   S{\scriptsize ET}
   S{\scriptsize ELECTION} problem under
   strict majority thresholds
   and
   focus on three popular network structures:
   cycle permutation graphs,
   generalized Petersen graphs
   and
   torus cordalis.

  %
  %
  %
  %
  Consider two identical disjoint copies $G_1$ and $G_2$ of a graph $G$
  with $p$ vertices $(p\geq 4)$,
  such that $V(G_1)=\{v_1,v_2,\ldots,v_p\}$
  and $V(G_2)=\{u_1,u_2,\ldots,u_p\}$,
  where $v_i$ and $u_i$ are corresponding vertices for each $i$.
  For a permutation $\pi$ on $\{1,2,\ldots,p\}$,
  the {\em $\pi$-permutation graph} of $G$, denoted by $P_\pi(G)$,
  consists of $G_1$ and $G_2$ along with $p$ additional edges
  $v_iu_{\pi(i)}$, $i=1,2,\ldots,p$.
  Note that the graph $P_\pi(G)$ depends not
  only on the choice of the permutation $\pi$ but
  on the particular labeling of $G$ as well.
  Permutation graphs were introduced in
  \cite{cycle-permutation-def,cycle-permutation-def-2}.
   The {\em $n$-path} $P_n$ is the graph  having
  $V(P_n)=\{x_1, x_2, \ldots, x_n\}$ and
  $E(P_n)=\{x_1x_2, x_2x_3, \ldots, x_{n-1}x_n\}$.
  The {\em $n$-cycle} $C_n$ is the graph  having
  $V(C_n)=V(P_n)$ and
  $E(C_n)=E(P_n)\cup\{ x_nx_1\}$.
  If $G$ is a cycle, then
  $P_\pi(G)$ is also called a
  {\em cycle permutation graph}.
  As examples, cycle permutation graphs $P_\pi(C_5)$
  are depicted in Figure 1.


 \begin{center}
 \includegraphics[scale=0.57]{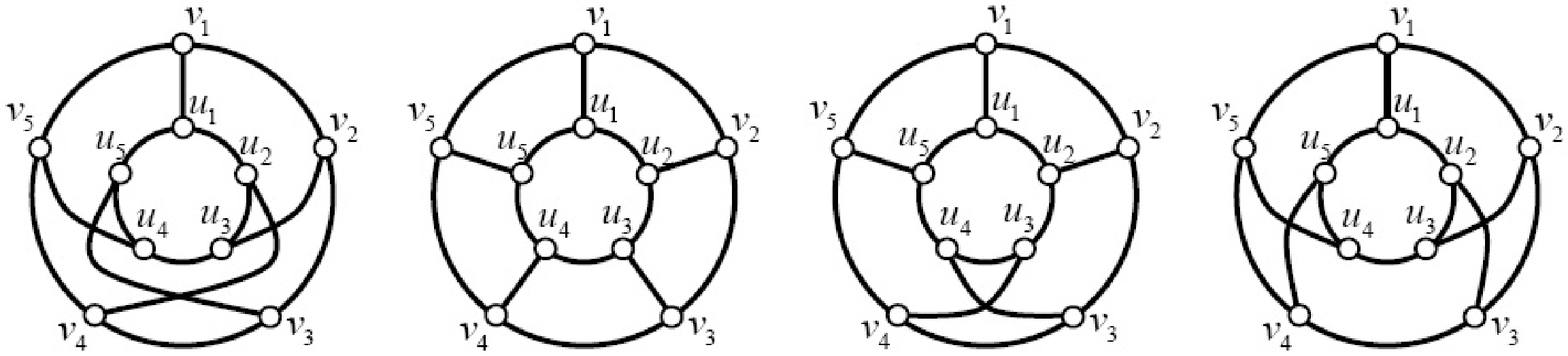}
 \begin{description}
    \item[Figure 1.]
    There are four  cycle permutation graphs for $P_\pi(C_5)$.
 \end{description}
 \end{center}

   %
   %
   %
   %
   For $m\geq 3$ and $1\leq s\leq \lfloor{m-1\over 2}\rfloor$,
   the {\em generalized Petersen graph} $P(m,s)$ is
   defined to be the graph with
   vertex set $$V(P(m,s))=\{v_1,v_2,\ldots,v_m,u_1,u_2,\ldots,u_m\}$$
   and edge set
   $$E(P(m,s))=\{v_iv_{i+1},u_iv_i,u_iu_{i+s}:i=1,2,\ldots,m\},$$
   where the subscripts are read modulo $m$.
   These graphs were introduced by Coxeter \cite{coxeter1950} and
   named by Watkins \cite{watkins1969}.
   As examples, generalized Petersen graphs $P(5,2)$, $P(10,2)$,
   and $P(10,4)$ are depicted in Figure 2.
   The connection between
   generalized Petersen graphs and
   cycle permutation graphs
   is given in
   \cite{gpg_iff_cycle-permutation}.
   By the results in \cite{gpg_iff_cycle-permutation}, we see
   that $P(10,4)$ is not a cycle permutation graph.
   Clearly, there are two cycle permutation graphs in Figure
   1 which are not generalized Petersen graphs.


 \begin{center}
 \includegraphics[scale=0.54]{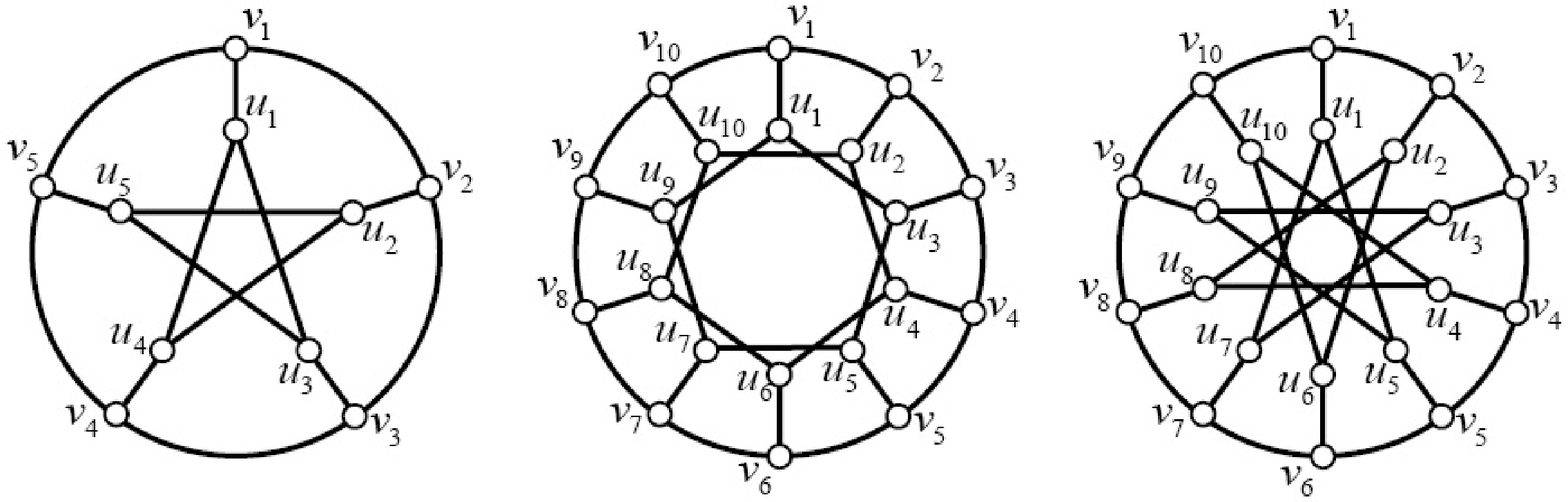}
 \begin{description}
    \item[Figure 2.]
   $P(5,2)\cong$ Petersen graph (left),
   $P(10,2)\cong$ dodecahedral graph (middle), and
   $P(10,4)$ (right).
 \end{description}
 \end{center}

 The $m\times n$ {\em toroidal mesh} $C_m\Box C_n$
 is the graph with vertex set
 $\{ (i,j): 1\leq i\leq m, 1\leq j\leq n\}$,
 where the neighbors of $(i,j)$ are
 $(i-1,j),(i+1,j),(i,j-1),(i,j+1)$.
 Here the arithmetic in the first coordinate is
 modulo $m$ and in the second coordinate modulo $n$.
 %
 The $m\times n$ {\em torus cordalis} $C_m\oslash C_n$
 and $m\times n$ toroidal mesh $C_m\Box C_n$ have the same vertex set.
 The edge set of $C_m\oslash C_n$
 is almost the same as $C_m\Box C_n$,
 except that the edge $(i,n)(i,1)$
 is replaced by the edge $(i,n)(i+1,1)$
 for  $1\leq i\leq m$.
 %
  The $m\times n$ {\em torus serpentinus} $C_m\otimes C_n$
 is almost the same as $C_m\oslash C_n$,
 except that the edge $(1,j)(m,j)$
 is replaced by the edge $(1,j)(m,j+1)$
 for  $1\leq j\leq n$.
 As examples, $C_4\Box C_3$,
   $C_4\oslash C_3$ and
   $C_4\otimes C_3$ are depicted in Figure 3.



 \begin{center}
 \includegraphics[scale=0.45]{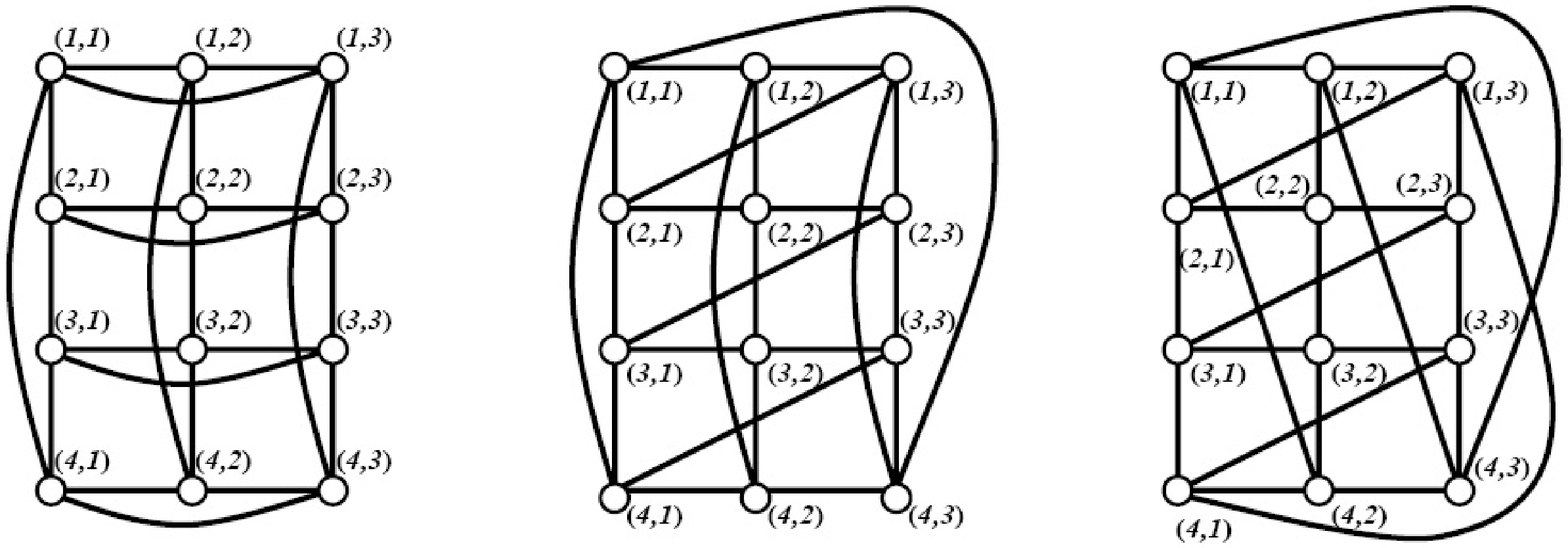}
 \begin{description}
    \item[Figure 3.]
   $C_4\Box C_3$ (left),
   $C_4\oslash C_3$ (middle), and
   $C_4\otimes C_3$ (right).
 \end{description}
 \end{center}

  %
  %
  %
  %
   In order to study the minimum seeds for $(G,\theta)$
   we introduce a sequential version of
   activation process in $(G,\theta)$,
   called {\em sequential activation process},
   which employs the following rule instead of the parallel updating
   rule:
  \begin{description}
          \item[Sequential updating rule:]
          Exactly one of
          inactive vertices that have at least $\theta(v)$
          already-active neighbors becomes active.
   \end{description}
    %
    %
   The proof of the following lemma is straightforward and
   so is omitted. In the sequel, Lemma \ref{seq=para}
   will be used without explicit reference to it.
   \begin{lemma}
   \label{seq=para}
   A minimum seed for $(G,\theta)$ under sequential updating
   rule is also a minimum seed for $(G,\theta)$ under parallel updating
   rule, and vice versa.
   \end{lemma}
  %
  %
   Consider a sequential activation process on $(G,\theta)$
   starting from a target set $S$.
   In this process, if $v_1,v_2,\ldots,v_r$ is the order that vertices
   in $[S]^G_\theta\setminus S$ become black,
   then $[v_1,v_2,\ldots,v_r]$
   is called the {\em convinced sequence}
   of $S$ on $(G,\theta)$.
   In order to describe a convinced sequence,
   we introduce an operation $\sqcup$
   on convinced subsequences.
   Let $\alpha=[v_1,v_2,\ldots,v_r]$
   and $\beta =[u_1,u_2,\ldots,u_s]$.
   Then $\alpha\sqcup \beta$ is defined as
   $$\alpha\sqcup \beta=[v_1,v_2,\ldots,v_r,u_1,u_2,\ldots,u_s],$$
   and for a list of convinced subsequences $\{\alpha_{i,j}\}_{1\leq i\leq \ell,
   1\leq j\leq k}$, the sequences $\sqcup_{i=1}^k\alpha_{i,j}$ and
   $\sqcup_{j=1}^\ell\sqcup_{i=1}^k \alpha_{i,j}$ are
   defined to be
   $$\sqcup_{i=1}^k\alpha_{i,j}=\alpha_{1,j}\sqcup\alpha_{2,j}\sqcup
   \cdots\sqcup\alpha_{k,j} \mbox{ and }
   \sqcup_{j=1}^\ell\sqcup_{i=1}^k \alpha_{i,j}=
   \sqcup_{j=1}^\ell(\sqcup_{i=1}^k \alpha_{i,j}).$$

   %
   %
   In Section \ref{CP-GP}, we precisely determine an optimal target set  for a social network $(G,\theta)$ when $G$ is a
   cycle permutation graph, and when $G$ is a generalized Petersen graph.
   In \cite{Flocchini2004}, Flocchini {\em et al.} constructed
   the following
   bounds on the size of a minimum seed for
   toroidal mesh $C_m \Box C_n$, torus cordalis $C_m \oslash C_n$,
    and  torus serpentinus $C_m \otimes C_n$ under
    strict majority thresholds.
    \begin{theorem}[\cite{Flocchini2004}]
    \label{bounds}
    (a) $\lceil{mn+1 \over 3}\rceil\leq \mbox{\rm min-seed}(C_m \oslash C_n,3)\leq
    \lceil{m \over 3}\rceil (n+1).$
    (b) If $G$ is a $C_m \Box C_n$ or a $C_m \otimes C_n$, then
    $\lceil{mn+1 \over 3}\rceil\leq \mbox{\rm min-seed}(G,3)\leq
    \min\{\lceil{m \over 3}\rceil (n+1), \lceil{n \over 3}\rceil
    (m+1)\}.$
    \end{theorem}

    In Section \ref{torus-cordalis} of this paper, we present
    some
    improved upper bounds and exact values for
    the parameter $\mbox{\rm min-seed}(C_m \oslash C_n,3)$.
    These results are summarized in Table 1.

 %
 %
 \begin{center}
 \begin{tabular}{|l|l|l|l|} \hline\hline
 {\em Theorems}                   & $m\geq 10$ &  $n\geq 6$
 & $\Phi$ \\ \hline
  Theorem \ref{TC-mx3s-a-b-c}(a)
    & odd   &   $n\equiv 0$ (mod 3)   &
    $\Phi={mn\over 3}+1$       \\ \hline
  Theorem \ref{TC-mx3s-a-b-c}(b)
    & even  &   $n\equiv 0$ (mod 6)  &
    $\Phi={mn\over 3}+1$       \\ \hline
  Theorem \ref{TC-mx3s-a-b-c}(c)
    & even  &   $n\equiv 3$ (mod 6)  &
    $\Phi\in \{{mn\over 3}+1, {mn\over 3}+2\}$     \\ \hline
  Theorem \ref{TC-mx3s+1}
    &       &   $n\equiv 1$ (mod 3)  &
    $\Phi\leq {mn\over 3}+{m\over 6}+1$     \\  \hline
  Theorem \ref{TC-mx3s+2}
    &       &   $n\equiv 2$ (mod 3)  &
    $\Phi\leq {mn\over 3}+{m\over 12}+{3\over 2}$     \\   \hline
  Theorem \ref{TC-3txn}
    &  $m\equiv 0$ (mod 3)     &     &
    $\Phi= {mn\over 3}+1$     \\
 \hline\hline
 \end{tabular}\\
 \begin{description}
    \item[Table 1.] New bounds and exact values for
    $\mbox{\rm min-seed}(C_m \oslash C_n,3)$,  where $\Phi$ denotes
    the parameter
 $\mbox{\rm min-seed}(C_m \oslash C_n,3)$.
 \end{description}
 \end{center}

  %
  %
  %
  %
  %
  %
  \section{Cycle permutation graphs and generalized Petersen graphs}
  \label{CP-GP}
   For $X\subseteq V(G)$, let $G[X]$ denote the induced subgraph of $G$
   whose vertex set is $X$
   and whose edge set consists of all edges of $G$ which have both
   ends in $X$.  The number ${1\over |V(G)|}\sum_{v\in V(G)} d(v)$ is
   defined to be the {\em average degree} of $G$ and is denoted by
   $d(G)$.
   We remark that all lower bounds in Theorem \ref{bounds}
   follows immediately from Lemma \ref{lowerbound-for-target-set}.

  \begin{theorem}[\cite{erdos1964+1965}]
  \label{average-degree-minimum-degree}
  Every graph with average degree at least $2k$, where $k$ is a
  positive integer, has an induced subgraph with minimum degree at
  least $k+1$.
  \end{theorem}

  \begin{lemma}
  \label{lowerbound-for-target-set}
  Let $G$ be a graph with $n$ vertices, $m$ edges
  and maximum degree $\Delta$.
  If a target set $S\subseteq V(G)$ influences all vertices in
  the social network
  $(G,k)$, then $|S|\geq {m-(\Delta-k)n+1\over k}$.
  \end{lemma}
  \begin{proof}
  Let $\bar{S}=V(G)\setminus S$ and $G\setminus S=G[\bar{S}]$.
  Since $S$ influences
  all vertices in the social network $(G,k)$,
  the graph $G\setminus S$ has no induced subgraph
  with minimum degree at least $\Delta -k+1$.
  By Theorem \ref{average-degree-minimum-degree}, it follows that
  $$2(\Delta -k)>d(G\setminus S)={2|E(G\setminus S)|\over |V(G\setminus S)|}
  \geq {2(m-\Delta |S|)\over n-|S|},$$
  where the last inequality follows from the fact that
  if $e$ is an edge in $E(G)$ but not in $E(G\setminus S)$,
  then $e$ has an end in $S$. We conclude that
  $(\Delta-k)(n-|S|)\geq m-\Delta |S|+1$ which completes the proof of
  the theorem.
 \end{proof}


 \begin{theorem}
 \label{cycle-permutation}
 For any permutation $\pi$ and $n\geq 4$,
 $\mbox{\rm min-seed}(P_\pi(C_n),2)=\lceil{n+1\over 2}\rceil$.
 \end{theorem}
 \begin{proof} Let $G=P_\pi(C_n)$.
  Suppose that $G$ consists of
  two disjoint copies $G_1$ and $G_2$ of $C_n$,
  such that
  $V(G_1)=\{v_1,v_2,\ldots,v_n\}$,
  $V(G_2)=\{u_1,u_2,\ldots,u_n\}$,
  $E(G_1)=\{v_1v_2,v_2v_3,\ldots,v_{n-1}v_n, v_nv_1\}$,
  $E(G_2)=\{u_1u_2,u_2u_3,\ldots,u_{n-1}u_n, u_nu_1\}$ and
  $E(G)=E(G_1)\cup E(G_2) \cup \{v_iu_{\pi(i)}: i=1,2,\ldots,n\}$.
  Without loss of generality, we might assume that $u_1v_1\in E(G)$.
  Let $H=G[\{v_3,v_4,\ldots,v_n\}]$.
  In the proof of Proposition 2 of \cite{Dreyer+Roberts}, it is shown
  that there is a minimum seed $S'$ for
  the social network $(H,2)$
  such that $|S'|=\lceil{(n-1)/ 2}\rceil$ and $\{v_3,v_n\}\subseteq S'$.
  Let $[v_{i_1},v_{i_2},\ldots,v_{i_a}]$
  be the convinced sequence of $S'$ on $(H,2)$, where $a=n-2-|S'|$.

  Choose the target set $S=S'\cup \{u_1\}$ for $(G,2)$.
  It can easily be check that $S$ can influence all vertices of
  $V(G)\setminus S$ in the social network $(G,2)$
  by using
  $[v_{i_1},v_{i_2},\ldots,v_{i_a}]\sqcup
  [v_1,v_2]\sqcup [u_2,u_3,\ldots,u_n]$
  as the convinced sequence.
  It follows that
  min-seed$(G,2)\leq |S|=\lceil{(n+1)/ 2}\rceil$.
 Moreover,
 from Lemma \ref{lowerbound-for-target-set},
 it is easy to see that min-seed$(G,2)
 \geq 
 \left\lceil{(n+1)/2}\right\rceil.$
 This completes the proof of the theorem.
 \end{proof}

 Clearly, if gcd$(m,s)=1$, then the generalized Petersen graph $P(m,s)$
 is a cycle permutation
 graph, and we see at once that the following corollary holds.
 In Theorem \ref{GPG-m-s} we further show that Corollary
 \ref{GPG-m-s-coprimes} holds if we drop the hypothesis
 $\gcd(m,s)=1$.

 \begin{corollary}
 \label{GPG-m-s-coprimes}
 If $\gcd(m,s)=1$, then
 $\mbox{\rm min-seed}(P(m,s),2)=\lceil{m+1\over 2}\rceil$.
 \end{corollary}

 \begin{theorem}
 \label{GPG-m-s}
 For $m\geq 3$ and $1\leq s\leq \lfloor{m-1\over 2}\rfloor$,
 $\mbox{\rm min-seed}(P(m,s),2)=\lceil{m+1\over 2}\rceil$.
 \end{theorem}
 \begin{proof}
   Let $G=P(m,s)$.
   Suppose that
   $V(G)=\{v_1,v_2,\ldots,v_m,u_1,u_2,\ldots,u_m\}$
   and
   $E(G)=\{v_iv_{i+1},u_iv_i,u_iu_{i+s}:i=1,2,\ldots,m\}$,
   where the subscripts are read modulo $m$.
   Let $H$ be the graph
 $G[\{v_{s+1},v_{s+2},\ldots,v_{m-s}\}]$.
 Since $H$ is a $(m-2s)$-path, by the proof of Proposition 2 in
 \cite{Dreyer+Roberts}, we get a minimum seed $S'$ for
 the social network $(H,2)$ such that $|S'|=\lceil{(m-2s+1)/2}\rceil$
 and $\{v_{s+1},v_{m-s}\}\subseteq S'$.
 Let $[v_{j_1},v_{j_2},\ldots,v_{j_a}]$
  be the convinced sequence of $S'$ on $(H,2)$, where $a=m-2s-|S'|$.

 Consider $S=S'\cup \{u_1,u_2,\ldots,u_s\}$ as a target set
 for $(G,2)$.
 Since $u_{m-s+i}$ is adjacent to both
 $u_i$ and
 $u_{m-2s+i}$
 for $i\in \{1,2,\ldots,s\}$ and
 $u_{s+j}$ is adjacent to both
 $u_j$ and
 $v_{s+j}$ for $j\in \{1,2,\ldots,m-2s\}$,
 we see that $S$ can influence all vertices of $V(G)\setminus S$
 in the social network $(G,2)$ by using
 $[v_{j_1},v_{j_2},\ldots,v_{j_a}]\sqcup
  [v_s,v_{s-1},\ldots,v_1]\sqcup
  [u_{s+1},u_{s+2},\ldots,u_{m-s}]\sqcup
  [u_{m-s+1},u_{m-s+2},\ldots,u_m]\sqcup
  [v_{m-s+1},v_{m-s+2},\ldots,v_m]$
 as the convinced sequence.
 Therefore min-seed$(G,2)\leq |S|=|S'|+s=\lceil{(m+1)/ 2}\rceil$.
 Moreover,
 by Lemma \ref{lowerbound-for-target-set},
 it can be seen that min-seed$(G,2)
 \geq 
 \left\lceil{(m+1)/2}\right\rceil.$
 This complete the proof of the theorem.
\end{proof}

  %
  %
  %
  %
  %
  \section{Torus cordalis}
  \label{torus-cordalis}

  %
  %
  %
  \begin{theorem}
  \label{TC-mx3}
  $\mbox{\rm min-seed}(C_m \oslash C_3,3)=m+1$ for any $m\geq 3$.
  \end{theorem}
  \begin{proof}
  Let $G=C_m \oslash C_3$.
  First we show that
  $\mbox{\rm min-seed}(G,3)\leq m+1$ by giving a target
  set $S\subseteq V(G)$ which influences all vertices of $G$.
  Denote by $S_1$ and $S_2$ the sets
  $\{(2i+1,1): 0\leq i\leq \lfloor {m-1 \over 2}\rfloor\}$
  and $\{(2i,2): 1\leq i\leq \lfloor {m\over 2}\rfloor\}$,
  respectively.
  Set $S=S_1\cup S_2\cup \{(1,3)\}$.
  Let
  $\alpha_1= [(2,1), (4,1),\ldots, (2\lfloor\frac{m}{2}\rfloor,1)]$, $\alpha_2=[ (1,2), (3,2),\ldots, (1+2\lfloor\frac{m-1}{2}\rfloor,2)]$
  and
  $\alpha_3=[ (2,3), (3,3), \ldots, (m,3)]$.
  It is straightforward to see that
  $   \alpha_1 \sqcup    \alpha_2 \sqcup   \alpha_3$ is a convinced
   sequence of $S$ on $(G,3)$
   (see Figure 1 in Appendix for a graphical illustration
   of this convinced sequence).
   By the lower bound of $\mbox{\rm min-seed}(C_m \oslash C_3,3)$ in
   Theorem \ref{bounds}(a), we see that $\mbox{\rm min-seed}(C_m \oslash C_3,3)=m+1$.
  \end{proof}

  %
  %
  %
  %
  \begin{theorem}
  \label{TC-mx3s-a-b-c}
   Let $s\geq 2$ be an integer.
  (a) If  $m\geq 5$ is an odd integer, then
  $\mbox{\rm min-seed}(C_m \oslash C_{3s},3)=ms+1$.
  (b) If  $m\geq 8$ is an even integer and $s$ is an even integer, then
  $\mbox{\rm min-seed}(C_m \oslash C_{3s},3)=ms+1$.
  (c) If  $m\geq 8$ is an even integer and $s$ is an odd integer, then
  $\mbox{\rm min-seed}(C_m \oslash C_{3s},3)=ms+1$ or $ms+2$.
  \end{theorem}
  \begin{proof}
  Let $G=C_m \oslash C_{3s}$.

  (a) Denote by $S_1$ and $S_2$ the sets
  $\cup_{j=0}^{s-1}\{(1,1+3j), (2, 2+3j), (3, 1+3j), (4, 3+3j), (5, 2+3j)\}$
  and
  $\cup_{j=0}^{s-1}\cup_{i=0}^{m-7 \over 2} \{(6+2i,3+3j), (7+2i, 2+3j)\}$,
  respectively.
  Let $S=S_1\cup S_2 \cup \{(4,1)\}$.
  In the social network $(G,3)$,
  it is straightforward to check that the target set $S$ can influence
  all vertices in $V(G)\setminus S$ by using the convinced sequence
  $\alpha=\alpha_1\sqcup \alpha_2 \sqcup\alpha_3\sqcup \alpha_4
  \sqcup\alpha_5$, where
     $\alpha_1 = \sqcup_{j=0}^{s-1} [(1,2+3j), (2, 1+3j)]$,
     $\alpha_2 = \sqcup_{j=0}^{s-1}\sqcup_{i=0}^{m-7 \over 2}
                     [(5+2i,3+3j), (6+2i, 2+3j)]$,
     $\alpha_3 = [(5,1),(6,1),(7,1),\ldots, (m,1)]$,
     $\alpha_4 = [(4,2),(3,2),(3,3),(2,3),(1,3), (m,3)]$, and
     $\alpha_5 = \sqcup_{j=0}^{s-2}
   ([(m,4+3j),(m-1,4+3j),\ldots, (4,4+3j)]\sqcup
    [(4,5+3j),(3,5+3j),(3,6+3j),(2,6+3j),(1,6+3j),(m,6+3j)])$
    (see Figure 2 in Appendix for a graphical illustration
   of this convinced sequence $\alpha$).
    Therefore $\mbox{\rm min-seed}(C_m \oslash C_{3s},3)\leq
    |S|=ms+1$ and hence by Theorem \ref{bounds}(a), we have
    $\mbox{\rm min-seed}(C_m \oslash C_{3s},3)=
    ms+1$.

    (b) Denote by $S_1$, $S_2$ and $S_3$ the sets
  $\cup_{j=0}^{s-1}\{(1,1+3j), (2, 2+3j), (3, 1+3j), (4, 3+3j), (5,  2+3j)\}$,
  $\cup_{j=0}^{s-1}\cup_{i=0}^{m-10 \over 2} \{(6+2i,3+3j), (7+2i, 2+3j)\}$
  and
  $\cup_{j=0}^{s-1}\{(m-2,1+3j), (m-1, 3+3j), (m, 2+3j)\}$,
  respectively.
  Let $S=S_1\cup S_2 \cup S_3 \cup \{(4,1)\}$.
  It can readily be checked that the target set $S$ can influence
  all vertices in $V(G)\setminus S$ by using the convinced sequence
  $\alpha=\alpha_1\sqcup \alpha_2 \sqcup\alpha_3\sqcup \alpha_4
  \sqcup\alpha_5\sqcup\alpha_6\sqcup\alpha_7$
  (see Figure 3 in Appendix for a graphical illustration of this
  convinced sequence $\alpha$), where
  \begin{description}
    \item[$\alpha_1=$] $\sqcup_{j=0}^{s-1} [(1,2+3j), (2, 1+3j)]$,
    \item[$\alpha_2=$] $\sqcup_{j=0}^{s-1}\sqcup_{i=0}^{m-10 \over 2}
                     [(5+2i,3+3j), (6+2i, 2+3j)]$,
    \item[$\alpha_3=$] $[(5,1),(6,1),(7,1),\ldots, (m-3,1)]\sqcup [(m-3,3s)]$,
    \item[$\alpha_4=$] $[(4,2),(3,2),(3,3),(2,3),$ $(1,3), (m,3)]$,
    \item[$\alpha_5=$] $\sqcup_{k=0}^{s-4 \over 2}
   ([(m,4+6k), (m-1,4+6k), (m-1,5+6k), (m-2,5+6k), (m-2,6+6k), (m-3,6+6k)]\sqcup
    [(m-3,7+6k),(m-4,7+6k),\ldots,(4,7+6k)]\sqcup
    [(4,8+6k),(3,8+6k),(3,9+6k),(2,9+6k),(1,9+6k),(m,9+6k)])$,
    \item[$\alpha_6=$]
    $[(m,3s-2),(m-1,3s-2),(m-1,3s-1),(m-2,3s-1),(m-2,3s)]$, and
    \item[$\alpha_7=$] $\sqcup_{k=0}^{s-2 \over 2}
   ([(m,1+6k), (m-1,1+6k), (m-1,2+6k), (m-2,2+6k), (m-2,3+6k), (m-3,3+6k)]\sqcup
    [(m-3,4+6k),(m-4,4+6k),\ldots,(4,4+6k)]\sqcup
    [(4,5+6k),(3,5+6k),(3,6+6k),(2,6+6k),(1,6+6k),(m,6+6k)])$.
  \end{description}
    Therefore $\mbox{\rm min-seed}(C_m \oslash C_{3s},3)\leq
    |S|=ms+1$ and hence by Theorem \ref{bounds}(a), we have
    $\mbox{\rm min-seed}(C_m \oslash C_{3s},3)=
    ms+1$.

    (c) Denote by $S_1$, $S_2$ and $S_3$ the sets
  $\cup_{j=0}^{s-1}\{(1,1+3j), (2, 2+3j), (3, 1+3j), (4, 3+3j), (5,  2+3j)\}$,
  $\cup_{j=0}^{s-1}\cup_{i=0}^{m-10 \over 2} \{(6+2i,3+3j), (7+2i, 2+3j)\}$
  and
  $\cup_{j=0}^{s-1}\{(m-2,1+3j), (m-1, 3+3j), (m, 2+3j)\}$,
  respectively.
  Let $S=S_1\cup S_2 \cup S_3 \cup \{(4,1),(m-1,1)\}$.
 It is straightforward to check that the target set $S$ can influence
  all vertices in $V(G)\setminus S$ by using the convinced sequence
  $\alpha=\alpha_1\sqcup \alpha_2 \sqcup\alpha_3\sqcup \alpha_4
  \sqcup\alpha_5\sqcup\alpha_6\sqcup\alpha_7$
  (see Figure 4 in Appendix for a graphical illustration of this
  convinced sequence $\alpha$), where
    \begin{description}
        \item[$\alpha_1=$]
        $\sqcup_{j=0}^{s-1} [(1,2+3j), (2, 1+3j)]$,
        \item[$\alpha_2=$] $\sqcup_{j=0}^{s-1}\sqcup_{i=0}^{m-10 \over 2}
                     [(5+2i,3+3j), (6+2i, 2+3j)]$,
        \item[$\alpha_3=$] $[(5,1),(6,1),(7,1),\ldots, (m-3,1)]$,
        \item[$\alpha_4=$] $[(4,2),(3,2),(3,3),(2,3),(1,3), (m,3)]$,
        \item[$\alpha_5=$] $[(m,1),(m-1,2),(m-2,2),(m-2,3),(m-3,3)]$,
        \item[$\alpha_6=$] $\sqcup_{k=0}^{s-3 \over 2}
   ([(m-3,4+6k), (m-4,4+6k),\ldots, (4,4+6k)]\sqcup
    [(4,5+6k),(3,5+6k),(3,6+6k),(2,6+6k),(1,6+6k),(m,6+6k)]\sqcup
    [(m,7+6k),(m-1,7+6k),(m-1,8+6k),(m-2,8+6k),(m-2,9+6k),(m-3,9+6k)])$, and
        \item[$\alpha_7=$] $\sqcup_{k=0}^{s-3 \over 2}
   ([(m,4+6k), (m-1,4+6k), (m-1,5+6k), (m-2,5+6k), (m-2,6+6k), (m-3,6+6k)]\sqcup
    [(m-3,7+6k),(m-4,7+6k),\ldots,(4,7+6k)]\sqcup
    [(4,8+6k),(3,8+6k),(3,9+6k),(2,9+6k),(1,9+6k),(m,9+6k)])$.
    \end{description}
    Therefore $\mbox{\rm min-seed}(C_m \oslash C_{3s},3)\leq
    |S|=ms+2$ and hence by Theorem \ref{bounds}(a), we have
    $\mbox{\rm min-seed}(C_m \oslash C_{3s},3)=
    ms+1$ or $ms+2$.
  \end{proof}

  %
  %
  %
  %
  \begin{theorem}
  \label{TC-mx3s+1}
   If  $m\geq 8$ and $n\equiv 1 \pmod{3}$, then
  $\mbox{\rm min-seed}(C_m \oslash C_n,3)\leq {mn\over 3}+{m\over 6}+1$.
  \end{theorem}
  \begin{proof}
  Let $G=C_m \oslash C_n$ and $n=3s+1$.
  The proof is divided into three cases,
  according to the parity of the two integers $m$ and $s$.

  {\bf Case 1.} $m\geq 5$ is odd.
  Let $S_1$, $S_2$ and $S_3$ denote the sets
  $\cup_{j=0}^{s-1}\{(1,2+3j), (2, 3+3j), (3, 2+3j), (4, 4+3j), (5, 3+3j)\}$,
  $\{(6,1), (8,1),\ldots, (m-1,1)\}$
  and
  $\cup_{j=0}^{s-1}\cup_{i=0}^{m-7 \over 2} \{(6+2i,4+3j), (7+2i, 3+3j)\}$,
  respectively.
  Let $S=\{(1,1), (2,1), (4,1)\} \cup S_1\cup S_2\cup S_3  $.
  It is straightforward to check that the target set $S$ can influence
  all vertices in $V(G)\setminus S$ by using the convinced sequence
  $\alpha=\alpha_1\sqcup \alpha_2 \sqcup\alpha_3\sqcup \alpha_4
  $, where
     $\alpha_1 = \sqcup_{j=0}^{s-1} [(1,3+3j), (2, 2+3j)]$,
     $\alpha_2 = \sqcup_{j=0}^{s-1}\sqcup_{i=0}^{m-7 \over 2}
                     [(5+2i,4+3j), (6+2i, 3+3j)]$,
     $\alpha_3 = [(3,1),(5,1),(7,1),\ldots, (m,1)]$, and
     $\alpha_4 = \sqcup_{j=0}^{s-1}
   ([(m,2+3j),(m-1,2+3j),\ldots, (4,2+3j)]\sqcup
    [(4,3+3j),(3,3+3j),(3,4+3j),(2,4+3j),(1,4+3j),(m,4+3j)])$
    (see Figure 5 in Appendix for a graphical illustration
   of this convinced sequence $\alpha$).
    Therefore $\mbox{\rm min-seed}(C_m \oslash C_n,3)\leq
    |S|={mn\over 3}+{m\over 6}+{1\over 2}$.

  {\bf Case 2.} $m\geq 8$ is even and $s$ is odd.
  Denote by $S_1$, $S_2$, $S_3$ and $S_4$ the sets
  $\cup_{j=0}^{s-1}\{(1,2+3j), (2, 3+3j), (3, 2+3j), (4, 4+3j), (5, 3+3j)\}$,
  $\{(6,1), (8,1),\ldots, (m-4,1)\}$,
  $\cup_{j=0}^{s-1}\cup_{i=0}^{m-10 \over 2} \{(6+2i,4+3j), (7+2i, 3+3j)\}$
  and
  $\cup_{j=0}^{s-1}\{(m-2,2+3j), (m-1, 4+3j), (m, 3+3j)\}$,
  respectively.
  Let $S=S_1\cup S_2 \cup S_3 \cup S_4\cup \{(1,1), (2,1), (4,1), (m-1,1)\}$.
  It can readily be checked that the target set $S$ can influence
  all vertices in $V(G)\setminus S$ by using the convinced sequence
  $\alpha=\alpha_1\sqcup \alpha_2 \sqcup\alpha_3\sqcup \alpha_4
  \sqcup\alpha_5\sqcup\alpha_6\sqcup\alpha_7$
  (see Figure 6 in Appendix for a graphical illustration of this
  convinced sequence $\alpha$), where
  \begin{description}
    \item[$\alpha_1=$] $\sqcup_{j=0}^{s-1} [(1,3+3j), (2, 2+3j)]$,
    \item[$\alpha_2=$] $\sqcup_{j=0}^{s-1}\sqcup_{i=0}^{m-10 \over 2}
                     [(5+2i,4+3j), (6+2i, 3+3j)]$,
    \item[$\alpha_3=$] $[(3,1),(5,1),(7,1),\ldots, (m-5,1)]$,
    \item[$\alpha_4=$] $[(m,1),(m,2),(m-1,2),(m-1,3),(m-2,3), (m-2,4), (m-3,4)]$,
    \item[$\alpha_5=$] $\sqcup_{k=0}^{s-3 \over 2}
   ([(m-3,5+6k), (m-4,5+6k),\ldots , (4,5+6k)]\sqcup
    [(4,6+6k),(3,6+6k),(3,7+6k),(2,7+6k),(1,7+6k),(m,7+6k)]\sqcup
    [(m,8+6k),(m-1,8+6k),(m-1,9+6k),(m-2,9+6k),(m-2,10+6k),(m-3,10+6k)])$,
    \item[$\alpha_6=$]
    $[(m-2,1),(m-3,1)]\sqcup
    [(m-3,2), (m-4,2),\ldots, (4,2)]\sqcup
    [(4,3),(3,3),(3,4),$ $(2,4),(1,4),(m,4)]$, and
    \item[$\alpha_7=$] $\sqcup_{k=0}^{s-3 \over 2}
   ([(m,5+6k), (m-1,5+6k), (m-1,6+6k), (m-2,6+6k), (m-2,7+6k), (m-3,7+6k)]\sqcup
    [(m-3,8+6k),(m-4,8+6k),\ldots,(4,8+6k)]\sqcup
    [(4,9+6k),(3,9+6k),(3,10+6k),(2,10+6k),(1,10+6k),(m,10+6k)])$.
  \end{description}
    Therefore $\mbox{\rm min-seed}(C_m \oslash C_n,3)\leq
    |S|={mn\over 3}+{m\over 6}$.

  {\bf Case 3.} $m\geq 8$ and $s$ are both even.
  Denote by $S_1$, $S_2$, $S_3$ and $S_4$ the sets
  $\cup_{j=0}^{s-1}\{(1,2+3j), (2, 3+3j), (3, 2+3j), (4, 4+3j), (5, 3+3j)\}$,
  $\{(6,1), (8,1),\ldots, (m-4,1)\}$,
  $\cup_{j=0}^{s-1}\cup_{i=0}^{m-10 \over 2} \{(6+2i,4+3j), (7+2i, 3+3j)\}$
  and
  $\cup_{j=0}^{s-1}\{(m-2,2+3j), (m-1, 4+3j), (m, 3+3j)\}$,
  respectively.
  Let $S=S_1\cup S_2 \cup S_3 \cup S_4\cup \{(1,1), (2,1), (4,1), (m-2,1),(m-1,1)\}$.
  It is straightforward to check that the target set $S$ can influence
  all vertices in $V(G)\setminus S$ by using the convinced sequence
  $\alpha=\alpha_1\sqcup \alpha_2 \sqcup\alpha_3\sqcup \alpha_4
  \sqcup\alpha_5$
  (see Figure 7 in Appendix for a graphical illustration of this
  convinced sequence $\alpha$), where
    \begin{description}
        \item[$\alpha_1=$]
        $\sqcup_{j=0}^{s-1} [(1,3+3j), (2, 2+3j)]$,
        \item[$\alpha_2=$] $\sqcup_{j=0}^{s-1}\sqcup_{i=0}^{m-10 \over 2}
                     [(5+2i,4+3j), (6+2i, 3+3j)]$,
        \item[$\alpha_3=$] $[(3,1),(5,1),(7,1),\ldots, (m-3,1)]\sqcup [(m,1)]$,
        \item[$\alpha_4=$] $\sqcup_{k=0}^{s-2 \over 2}
   ([(m-3,2+6k), (m-4,2+6k),\ldots, (4,2+6k)]\sqcup
    [(4,3+6k),(3,3+6k),(3,4+6k),(2,4+6k),(1,4+6k),(m,4+6k)]\sqcup
    [(m,5+6k),(m-1,5+6k),(m-1,6+6k),(m-2,6+6k),(m-2,7+6k),(m-3,7+6k)])$, and
        \item[$\alpha_5=$] $\sqcup_{k=0}^{s-2 \over 2}
   ([(m,2+6k), (m-1,2+6k), (m-1,3+6k), (m-2,3+6k), (m-2,4+6k), (m-3,4+6k)]\sqcup
    [(m-3,5+6k),(m-4,5+6k),\ldots,(4,5+6k)]\sqcup
    [(4,6+6k),(3,6+6k),(3,7+6k),(2,7+6k),(1,7+6k),(m,7+6k)])$.
    \end{description}
    We conclude that $\mbox{\rm min-seed}(C_m \oslash C_n,3)\leq
    |S|={mn\over 3}+{m\over 6}+1$. This completes the proof of the
    theorem.
  \end{proof}

  %
  %
  %
  %
  \begin{theorem}
  \label{TC-mx3s+2}
   If  $m\geq 10$, $n\equiv 2 \pmod{3}$ and $n\geq 5$, then
  $\mbox{\rm min-seed}(C_m \oslash C_n,3)\leq
  {mn\over 3}+{m\over 12}+{3\over 2}$.
  \end{theorem}

  \begin{proof}
  Let $G=C_m \oslash C_n$, $m=4t+ r$ and  $n=3s+2$, where
  $t,r,s$ are integers with $0\leq r\leq 3$.
  The proof is divided into six cases,
  according to the value of $r$ and the parity of $s$.

  {\bf Case 1.} $r=0$ and $s$ is even. In this case,
  let $S_1$, $S_2$, $S_3$ and $S_4$ denote the sets
  $\cup_{j=0}^{s-1}\{(1, 3+3j), (2, 4+3j), (3, 3+3j)\}$,
  $\cup_{i=0}^{t-3}\{(4+4i,1), (6+4i,1), (6+4i,2)\}$,
  $\cup_{j=0}^{s-1}\cup_{i=0}^{t-3} \{(4+4i,5+3j), (5+4i, 3+3j), (6+4i, 5+3j), (7+4i, 3+3j)\}$
  and
  $\cup_{j=0}^{s-1}\{(m-4,5+3j), (m-3, 4+3j), (m-2, 3+3j), (m-1, 5+3j), (m,4+3j)\}$,
  respectively.
  Let $S=\{(2,1), (3,2), (m-4,1), (m-3,2), (m-1,1), (m,2)\}
  \cup S_1\cup S_2 \cup S_3 \cup S_4$.
  It is straightforward to check that the target set $S$ can influence
  all vertices in $V(G)\setminus S$ by using the convinced sequence
  $\alpha=\alpha_1\sqcup \alpha_2 \sqcup\alpha_3\sqcup \alpha_4
  \sqcup\alpha_5\sqcup\alpha_6\sqcup\alpha_7\sqcup\alpha_8$
  (see Figure 8 in Appendix for a graphical illustration of this
  convinced sequence $\alpha$), where
  \begin{description}
    \item[$\alpha_1=$] $\sqcup_{j=0}^{s-1} [(1,4+3j), (2, 3+3j)]$,
    \item[$\alpha_2=$] $\sqcup_{i=0}^{t-3} [(7+4i,1),(7+4i,2),(5+4i,1),(5+4i,2),(4+4i,2)]$,
    \item[$\alpha_3=$] $\sqcup_{j=0}^{s-1}\sqcup_{i=0}^{t-3}
                     [(4+4i,3+3j), (5+4i,5+3j), (6+4i,3+3j), (7+4i,5+3j)]$,
    \item[$\alpha_4=$] $[(m-4,2),(m-3,1)]$,
    \item[$\alpha_5=$] $\sqcup_{k=0}^{s-2 \over 2}
   ([(m-3,3+6k), (m-4,3+6k)]\sqcup
    [(m-4,4+6k),(m-5,4+6k),(m-6,4+6k),\ldots, (3,4+6k)]\sqcup
    [(3,5+6k), (2,5+6k), (1,5+6k), (m,5+6k)]\sqcup
    [(m,6+6k),(m-1,6+6k),(m-1,7+6k),(m-2,7+6k),(m-2,8+6k),(m-3,8+6k)])$,
    \item[$\alpha_6=$] $[(m-2,1),(m-2,2),(m-1,2),(m,1)]$,
    \item[$\alpha_7=$]
    $[(3,1),(2,2),(1,2),(1,1)]$, and
    \item[$\alpha_8=$] $\sqcup_{k=0}^{s-2 \over 2}
   ([(m,3+6k), (m-1,3+6k), (m-1,4+6k), (m-2,4+6k), (m-2,5+6k), (m-3,5+6k), (m-3,6+6k), (m-4,6+6k)]\sqcup
    [(m-4,7+6k),(m-5,7+6k),(m-6,7+6k),\ldots,(3,7+6k)]\sqcup
    [(3,8+6k),(2,8+6k),(1,8+6k),(m,8+6k)])$.
  \end{description}
    We conclude that $\mbox{\rm min-seed}(C_m \oslash C_n,3)\leq
    |S|={mn\over 3}+{m\over 12}$.

  {\bf Case 2.} $r=0$ and $s$ is odd. In this case, let
   $S_1$, $S_2$, $S_3$ and $S_4$ denote the sets
  $\cup_{j=0}^{s-1}\{(1, 3+3j), (2, 4+3j), (3, 3+3j)\}$,
  $\cup_{i=0}^{t-3}\{(4+4i,1), (6+4i,1), (6+4i,2)\}$,
  $\cup_{j=0}^{s-1}\cup_{i=0}^{t-3} \{(4+4i,5+3j), (5+4i, 3+3j), (6+4i, 5+3j), (7+4i, 3+3j)\}$
  and
  $\cup_{j=0}^{s-1}\{(m-4,5+3j), (m-3, 4+3j), (m-2, 3+3j), (m-1, 5+3j), (m,4+3j)\}$,
  respectively.
  Let $S=\{(2,1), (3,2), (m-4,1), (m-3,2), (m-2,1), (m-1,1), (m,2)\}
  \cup S_1\cup S_2 \cup S_3 \cup S_4$.
  It is straightforward to check that the target set $S$ can influence
  all vertices in $V(G)\setminus S$ by using the convinced sequence
  $\alpha=\alpha_1\sqcup \alpha_2 \sqcup\alpha_3\sqcup \alpha_4
  \sqcup\alpha_5\sqcup \alpha_6 \sqcup\alpha_7\sqcup \alpha_8 \sqcup\alpha_9$
  (see Figure 9 in Appendix for a graphical illustration of this
  convinced sequence $\alpha$), where
    \begin{description}
    \item[$\alpha_1=$] $\sqcup_{j=0}^{s-1} [(1,4+3j), (2, 3+3j)]$,
    \item[$\alpha_2=$] $\sqcup_{i=0}^{t-3} [(7+4i,1),(7+4i,2),(5+4i,1),(5+4i,2),(4+4i,2)]$,
    \item[$\alpha_3=$] $\sqcup_{j=0}^{s-1}\sqcup_{i=0}^{t-3}
                     [(4+4i,3+3j), (5+4i,5+3j), (6+4i,3+3j), (7+4i,5+3j)]$,
    \item[$\alpha_4=$] $[(m-4,2),(m-3,1),(m-2,2),(m-1,2),(m,1)]$,
    \item[$\alpha_5=$] $[(3,1),(2,2),(1,2),(1,1)]$,
    \item[$\alpha_6=$] $[(m-3,3), (m-4,3)]\sqcup
    [(m-4,4),(m-5,4),(m-6,4),\ldots, (3,4)]\sqcup
    [(3,5),$\\$(2,5), (1,5), (m,5)]$,
    \item[$\alpha_7=$] $[(m,3),(m-1,3),(m-1,4),(m-2,4),(m-2,5),(m-3,5)]$,
    \item[$\alpha_8=$] $\sqcup_{k=0}^{s-3 \over 2}
   ([(m-3,6+6k), (m-4,6+6k)]\sqcup
    [(m-4,7+6k),(m-5,7+6k),(m-6,7+6k),\ldots, (3,7+6k)]\sqcup
    [(3,8+6k), (2,8+6k), (1,8+6k), (m,8+6k)]\sqcup
    [(m,9+6k),(m-1,9+6k),(m-1,10+6k),(m-2,10+6k),(m-2,11+6k),(m-3,11+6k)])$,
    and
    \item[$\alpha_9=$] $\sqcup_{k=0}^{s-3 \over 2}
   ([(m,6+6k), (m-1,6+6k), (m-1,7+6k), (m-2,7+6k), (m-2,8+6k), (m-3,8+6k), (m-3,9+6k), (m-4,9+6k)]\sqcup
    [(m-4,10+6k),(m-5,10+6k),(m-6,10+6k),\ldots,(3,10+6k)]\sqcup
    [(3,11+6k),(2,11+6k),(1,11+6k),(m,11+6k)])$.
    \end{description}
    It follows that $\mbox{\rm min-seed}(C_m \oslash C_n,3)\leq
    |S|={mn\over 3}+{m\over 12}+1$.

  {\bf Case 3.} $r=1$.
  In this case, let $S_1$, $S_2$, $S_3$ and $S_4$ denote the sets
  $\cup_{j=0}^{s-1}\{(1, 3+3j), (2, 4+3j), (3, 3+3j)\}$,
  $\cup_{i=0}^{t-2}\{(4+4i,1), (6+4i,1), (6+4i,2)\}$,
  $\cup_{j=0}^{s-1}\cup_{i=0}^{t-2} \{(4+4i,5+3j), (5+4i, 3+3j), (6+4i, 5+3j), (7+4i, 3+3j)\}$
  and
  $\cup_{j=0}^{s-1}\{(m-1, 5+3j), (m,4+3j)\}$,
  respectively.
  Let $S=\{(2,1), (3,2), (m-1,1), (m,2)\}
  \cup S_1\cup S_2 \cup S_3 \cup S_4$.
  It is straightforward to check that the target set $S$ can influence
  all vertices in $V(G)\setminus S$ by using the convinced sequence
  $\alpha=\alpha_1\sqcup \alpha_2 \sqcup\alpha_3\sqcup \alpha_4
  \sqcup\alpha_5\sqcup \alpha_6$
  (see Figure 10 in Appendix for a graphical illustration of this
  convinced sequence $\alpha$), where
    \begin{description}
    \item[$\alpha_1=$] $\sqcup_{j=0}^{s-1} [(1,4+3j), (2, 3+3j)]$,
    \item[$\alpha_2=$] $\sqcup_{i=0}^{t-2} [(7+4i,1),(7+4i,2),(5+4i,1),(5+4i,2),(4+4i,2)]$,
    \item[$\alpha_3=$] $\sqcup_{j=0}^{s-1}\sqcup_{i=0}^{t-2}
                     [(4+4i,3+3j), (5+4i,5+3j), (6+4i,3+3j), (7+4i,5+3j)]$,
    \item[$\alpha_4=$] $[(m-1,2),(m,1)]$,
    \item[$\alpha_5=$] $[(3,1),(2,2),(1,2),(1,1)]$, and
    \item[$\alpha_6=$] $\sqcup_{j=0}^{s-1}
   ([(m,3+3j), (m-1,3+3j)]\sqcup
    [(m-1,4+3j),(m-2,4+3j),(m-3,4+3j),\ldots, (3,4+3j)]\sqcup
    [(3,5+3j), (2,5+3j), (1,5+3j), (m,5+3j)])$.
    \end{description}
    Therefore $\mbox{\rm min-seed}(C_m \oslash C_n,3)\leq
    |S|={mn\over 3}+{m\over 12}+{1\over 4}$.

  {\bf Case 4.} $r=2$ and $s$ is even.
  In this case, let $S_1$, $S_2$, $S_3$ and $S_4$ denote the sets
  $\cup_{j=0}^{s-1}\{(1, 3+3j), (2, 4+3j), (3, 3+3j), (4, 5+3j), (5, 3+3j)\}$,
  $\cup_{i=0}^{t-3}\{(6+4i,1), (8+4i,1), (8+4i,2)\}$,
  $\cup_{j=0}^{s-1}\cup_{i=0}^{t-3} \{(6+4i,5+3j), (7+4i, 3+3j), (8+4i, 5+3j), (9+4i, 3+3j)\}$
  and
  $\cup_{j=0}^{s-1}\{(m-4,5+3j), (m-3, 4+3j), (m-2, 3+3j), (m-1, 5+3j), (m,4+3j)\}$,
  respectively.
  Let $S=\{(2,1), (3,2), (4,1), (5,2), (m-4,1), (m-3,2), (m-1,1), (m,2)\}
  \cup S_1\cup S_2 \cup S_3 \cup S_4$.
  It can readily be checked that the target set $S$ can influence
  all vertices in $V(G)\setminus S$ by using the convinced sequence
  $\alpha=\alpha_1\sqcup \alpha_2 \sqcup\alpha_3\sqcup \alpha_4
  \sqcup\alpha_5\sqcup\alpha_6\sqcup\alpha_7\sqcup\alpha_8
  \sqcup\alpha_9\sqcup\alpha_{10}$
  (see Figure 11 in Appendix for a graphical illustration of this
  convinced sequence $\alpha$), where
  \begin{description}
    \item[$\alpha_1=$] $\sqcup_{j=0}^{s-1} [(1,4+3j), (2, 3+3j)]$,
    \item[$\alpha_2=$] $[(4,2),(5,1)]$,
    \item[$\alpha_3=$] $\sqcup_{i=0}^{t-3} [(9+4i,1),(9+4i,2),(7+4i,1),(7+4i,2),(6+4i,2)]$,
    \item[$\alpha_4=$] $\sqcup_{j=0}^{s-1} [(4,3+3j), (5,5+3j)]$,
    \item[$\alpha_5=$] $\sqcup_{j=0}^{s-1}\sqcup_{i=0}^{t-3}
    [(6+4i,3+3j), (7+4i,5+3j), (8+4i,3+3j), (9+4i,5+3j)]$,
    \item[$\alpha_6=$] $[(m-4,2),(m-3,1)]$,
    \item[$\alpha_7=$] $\sqcup_{k=0}^{s-2 \over 2}
   ([(m-3,3+6k), (m-4,3+6k)]\sqcup
    [(m-4,4+6k),(m-5,4+6k),(m-6,4+6k),\ldots, (3,4+6k)]\sqcup
    [(3,5+6k), (2,5+6k), (1,5+6k), (m,5+6k)]\sqcup
    [(m,6+6k),(m-1,6+6k),(m-1,7+6k),(m-2,7+6k),(m-2,8+6k),(m-3,8+6k)])$,
    \item[$\alpha_8=$] $[(m-2,1),(m-2,2),(m-1,2),(m,1)]$,
    \item[$\alpha_9=$]
    $[(3,1),(2,2),(1,2),(1,1)]$, and
    \item[$\alpha_{10}=$] $\sqcup_{k=0}^{s-2 \over 2}
   ([(m,3+6k), (m-1,3+6k), (m-1,4+6k), (m-2,4+6k), (m-2,5+6k), (m-3,5+6k), (m-3,6+6k), (m-4,6+6k)]\sqcup
    [(m-4,7+6k),(m-5,7+6k),(m-6,7+6k),\ldots,(3,7+6k)]\sqcup
    [(3,8+6k),(2,8+6k),(1,8+6k),(m,8+6k)])$.
  \end{description}
    It follows that $\mbox{\rm min-seed}(C_m \oslash C_n,3)\leq
    |S|={mn\over 3}+{m\over 12}+{1\over 2}$.

  {\bf Case 5.} $r=2$ and $s$ is odd.
  In this case, let $S_1$, $S_2$, $S_3$ and $S_4$ denote the sets
  $\cup_{j=0}^{s-1}\{(1, 3+3j), (2, 4+3j), (3, 3+3j), (4, 5+3j), (5, 3+3j)\}$,
  $\cup_{i=0}^{t-3}\{(6+4i,1), (8+4i,1), (8+4i,2)\}$,
  $\cup_{j=0}^{s-1}\cup_{i=0}^{t-3} \{(6+4i,5+3j), (7+4i, 3+3j), (8+4i, 5+3j), (9+4i, 3+3j)\}$
  and
  $\cup_{j=0}^{s-1}\{(m-4,5+3j), (m-3, 4+3j), (m-2, 3+3j), (m-1, 5+3j), (m,4+3j)\}$,
  respectively.
  Let $S=\{(2,1), (3,2), (4,1), (5,2), (m-4,1), (m-3,2), (m-2,1), (m-1,1), (m,2)\}
  \cup S_1\cup S_2 \cup S_3 \cup S_4$.
  It is straightforward to check that the target set $S$ can influence
  all vertices in $V(G)\setminus S$ by using the convinced sequence
  $\alpha=\alpha_1\sqcup \alpha_2 \sqcup\alpha_3\sqcup \alpha_4
  \sqcup\alpha_5\sqcup \alpha_6 \sqcup\alpha_7\sqcup \alpha_8
  \sqcup\alpha_9\sqcup\alpha_{10}\sqcup\alpha_{11}$
  (see Figure 12 in Appendix for a graphical illustration of this
  convinced sequence $\alpha$), where
    \begin{description}
    \item[$\alpha_1=$] $\sqcup_{j=0}^{s-1} [(1,4+3j), (2, 3+3j)]$,
    \item[$\alpha_2=$] $[(4,2),(5,1)]$,
    \item[$\alpha_3=$] $\sqcup_{i=0}^{t-3} [(9+4i,1),(9+4i,2),(7+4i,1),(7+4i,2),(6+4i,2)]$,
    \item[$\alpha_4=$] $\sqcup_{j=0}^{s-1} [(4,3+3j), (5,5+3j)]$,
    \item[$\alpha_5=$] $\sqcup_{j=0}^{s-1}\sqcup_{i=0}^{t-3}
    [(6+4i,3+3j), (7+4i,5+3j), (8+4i,3+3j), (9+4i,5+3j)]$,
    \item[$\alpha_6=$] $[(m-4,2),(m-3,1),(m-2,2),(m-1,2),(m,1)]$,
    \item[$\alpha_7=$] $[(3,1),(2,2),(1,2),(1,1)]$,
    \item[$\alpha_8=$] $[(m-3,3), (m-4,3)]\sqcup
    [(m-4,4),(m-5,4),(m-6,4),\ldots, (3,4)]\sqcup
    [(3,5),$\\$ (2,5), (1,5), (m,5)]$,
    \item[$\alpha_9=$] $[(m,3),(m-1,3),(m-1,4),(m-2,4),(m-2,5),(m-3,5)]$,
    \item[$\alpha_{10}=$] $\sqcup_{k=0}^{s-3 \over 2}
   ([(m-3,6+6k), (m-4,6+6k)]\sqcup
    [(m-4,7+6k),(m-5,7+6k),(m-6,7+6k),\ldots, (3,7+6k)]\sqcup
    [(3,8+6k), (2,8+6k), (1,8+6k), (m,8+6k)]\sqcup
    [(m,9+6k),(m-1,9+6k),(m-1,10+6k),(m-2,10+6k),(m-2,11+6k),(m-3,11+6k)])$,
    and
    \item[$\alpha_{11}=$] $\sqcup_{k=0}^{s-3 \over 2}
   ([(m,6+6k), (m-1,6+6k), (m-1,7+6k), (m-2,7+6k), (m-2,8+6k), (m-3,8+6k), (m-3,9+6k), (m-4,9+6k)]\sqcup
    [(m-4,10+6k),(m-5,10+6k),(m-6,10+6k),\ldots,(3,10+6k)]\sqcup
    [(3,11+6k),(2,11+6k),(1,11+6k),(m,11+6k)])$.
    \end{description}
    We conclude that $\mbox{\rm min-seed}(C_m \oslash C_n,3)\leq
    |S|={mn\over 3}+{m\over 12}+{3\over 2}$.

  {\bf Case 6.} $r=3$.
  In this case, let $S_1$, $S_2$, $S_3$ and $S_4$ denote the sets
  $\cup_{j=0}^{s-1}\{(1, 3+3j), (2, 4+3j), (3, 3+3j), (4, 5+3j), (5, 3+3j)\}$,
  $\cup_{i=0}^{t-2}\{(6+4i,1), (8+4i,1), (8+4i,2)\}$,
  $\cup_{j=0}^{s-1}\cup_{i=0}^{t-2} \{(6+4i,5+3j), (7+4i, 3+3j), (8+4i, 5+3j), (9+4i, 3+3j)\}$
  and
  $\cup_{j=0}^{s-1}\{(m-1, 5+3j), (m,4+3j)\}$,
  respectively.
  Let $S=\{(2,1), (3,2), (4,1), (5,2), (m-1,1), (m,2)\}
  \cup S_1\cup S_2 \cup S_3 \cup S_4$.
  It is straightforward to check that the target set $S$ can influence
  all vertices in $V(G)\setminus S$ by using the convinced sequence
  $\alpha=\alpha_1\sqcup \alpha_2 \sqcup\alpha_3\sqcup \alpha_4
  \sqcup\alpha_5\sqcup \alpha_6\sqcup \alpha_7\sqcup \alpha_8$
  (see Figure 13 in Appendix for a graphical illustration of this
  convinced sequence $\alpha$), where
    \begin{description}
    \item[$\alpha_1=$] $\sqcup_{j=0}^{s-1} [(1,4+3j), (2, 3+3j)]$,
    \item[$\alpha_2=$] $[(4,2),(5,1)]$,
    \item[$\alpha_3=$] $\sqcup_{i=0}^{t-2} [(9+4i,1),(9+4i,2),(7+4i,1),(7+4i,2),(6+4i,2)]$,
    \item[$\alpha_4=$] $\sqcup_{j=0}^{s-1} [(4,3+3j), (5,5+3j)]$,
    \item[$\alpha_5=$] $\sqcup_{j=0}^{s-1}\sqcup_{i=0}^{t-2}
    [(6+4i,3+3j), (7+4i,5+3j), (8+4i,3+3j), (9+4i,5+3j)]$,
    \item[$\alpha_6=$] $[(m-1,2),(m,1)]$,
    \item[$\alpha_7=$] $[(3,1),(2,2),(1,2),(1,1)]$, and
    \item[$\alpha_8=$] $\sqcup_{j=0}^{s-1}
   ([(m,3+3j), (m-1,3+3j)]\sqcup
    [(m-1,4+3j),(m-2,4+3j),(m-3,4+3j),\ldots, (3,4+3j)]\sqcup
    [(3,5+3j), (2,5+3j), (1,5+3j), (m,5+3j)])$.
    \end{description}
    Therefore $\mbox{\rm min-seed}(C_m \oslash C_n,3)\leq
    |S|={mn\over 3}+{m\over 12}+{3\over 4}$. This completes the proof of the
    theorem.
  \end{proof}

  %
  %
  %
  %
  %
  %
  \begin{theorem}
  \label{TC-3txn}
   If  $m\equiv 0 \pmod{3}$ and $n\geq 2$,  then
  $\mbox{\rm min-seed}(C_m \oslash C_n,3)={mn\over 3}+1$.
  \end{theorem}

  \begin{proof}
  Let $G=C_m \oslash C_n$ and $m=3t$. The proof is divided into two
  cases, according to the parity of $n$.

  {\bf Case 1.} $n$ is even.
  Denote by $S_1$, $S_2$ and $S_3$ the sets
  $\cup_{j=0}^{n-2\over 2}\{(1, 1+2j), (2, 2+2j)\}$,
  $\cup_{i=0}^{t-2}\{(4+3i, 2), (6+3i,1)\}$
  and
  $\cup_{j=0}^{n-4\over 2}\cup_{i=0}^{t-2} \{(4+3i, 4+2j),(5+3i, 3+2j)\}$,
  respectively.
  Let $S=S_1\cup S_2\cup S_3 \cup \{(3,1)\}$.
  It is straightforward to check that the target set $S$ can influence
  all vertices in $V(G)\setminus S$ by using the convinced sequence
  $\alpha=\alpha_1\sqcup \alpha_2 \sqcup\alpha_3\sqcup \alpha_4$, where
     $\alpha_1 = \sqcup_{j=0}^{n-2\over 2} [(2,1+2j), (1, 2+2j)]$,
     $\alpha_2 = \sqcup_{i=0}^{t-2}\sqcup_{j=0}^{n-4\over 2}
                     [(4+3i,3+2j), (5+3i, 4+2j)]$,
     $\alpha_3 = [(3,2),(3,3),(3,4),\ldots, (3,n)]$,
     and
     $\alpha_4 = \sqcup_{i=0}^{t-2}
   ([(4+3i,1),(5+3i,1),(5+3i,2)]\sqcup
    [(6+3i,2), (6+3i,3), (6+3i,4),\ldots, (6+3i,n)])$
    (see Figure 14 in Appendix for a graphical illustration
   of this convinced sequence $\alpha$).
    Therefore $\mbox{\rm min-seed}(C_m \oslash C_n,3)\leq
    |S|={mn\over 3}+1$. By Theorem \ref{bounds}(a), we conclude that
    $\mbox{\rm min-seed}(C_m \oslash C_n,3)=
    {mn\over 3}+1$.

      {\bf Case 2.} $n$ is odd.
      By Theorem \ref{TC-mx3},
      it suffices to consider only the case when $n\geq 5$.
  Denote by $S_1$, $S_2$ and $S_3$ the sets
  $\cup_{i=0}^{t-1}\{(1+3i, 1), (2+3i, 2), (3+3i, 3)\}$,
  $\cup_{j=0}^{n-5\over 2}\{ (1, 5+2j),(2, 4+2j)\}$
  and
  $\cup_{j=0}^{n-5\over 2}\cup_{i=0}^{t-2} \{(4+3i,4+2j), (5+3i, 5+2j)\}$,
  respectively.
  Let $S=S_1\cup S_2\cup S_3 \cup \{(1,3)\}$.
  It is straightforward to check that the target set $S$ can influence
  all vertices in $V(G)\setminus S$ by using the convinced sequence
  $\alpha=\alpha_1\sqcup \alpha_2 \sqcup\alpha_3\sqcup \alpha_4$
  (see Figure 15 in Appendix for a graphical illustration
   of this convinced sequence $\alpha$), where
  \begin{description}
    \item[$\alpha_1=$] $\sqcup_{j=0}^{n-3\over 2} [(2,1+2j), (1, 2+2j)]$,
    \item[$\alpha_2=$] $\sqcup_{i=0}^{t-2}\sqcup_{j=0}^{n-7\over 2}
                     [(4+3i,5+2j), (5+3i, 6+2j)]$,
    \item[$\alpha_3=$]
    $\sqcup_{i=0}^{t-2}([(m-3i,n),(m-3i,n-1),\ldots,(m-3i,4)]
     \sqcup [(m-1-3i,4),(m-1-3i,3),(m-2-3i,3),(m-2-3i,2),
             (m-3i,2),(m-3i,1),(m-1-3i,1),(m-2-3i,n)])$, and
    \item[$\alpha_4=$] $[(3,2),(3,1),(2,n)]\sqcup[(3,n),(3,n-1),\ldots, (3,4)]$.
  \end{description}
    Therefore $\mbox{\rm min-seed}(C_m \oslash C_n,3)\leq
    |S|={mn\over 3}+1$, and hence by Theorem \ref{bounds}(a)
    we get
    $\mbox{\rm min-seed}(C_m \oslash C_n,3)=
    {mn\over 3}+1$. This completes the proof of the theorem.
  \end{proof}



 %
 %

\newpage

\noindent {\bf Appendix}: [Not for publication - for referees'
information only]

 \begin{center}
 \includegraphics[scale=0.6]{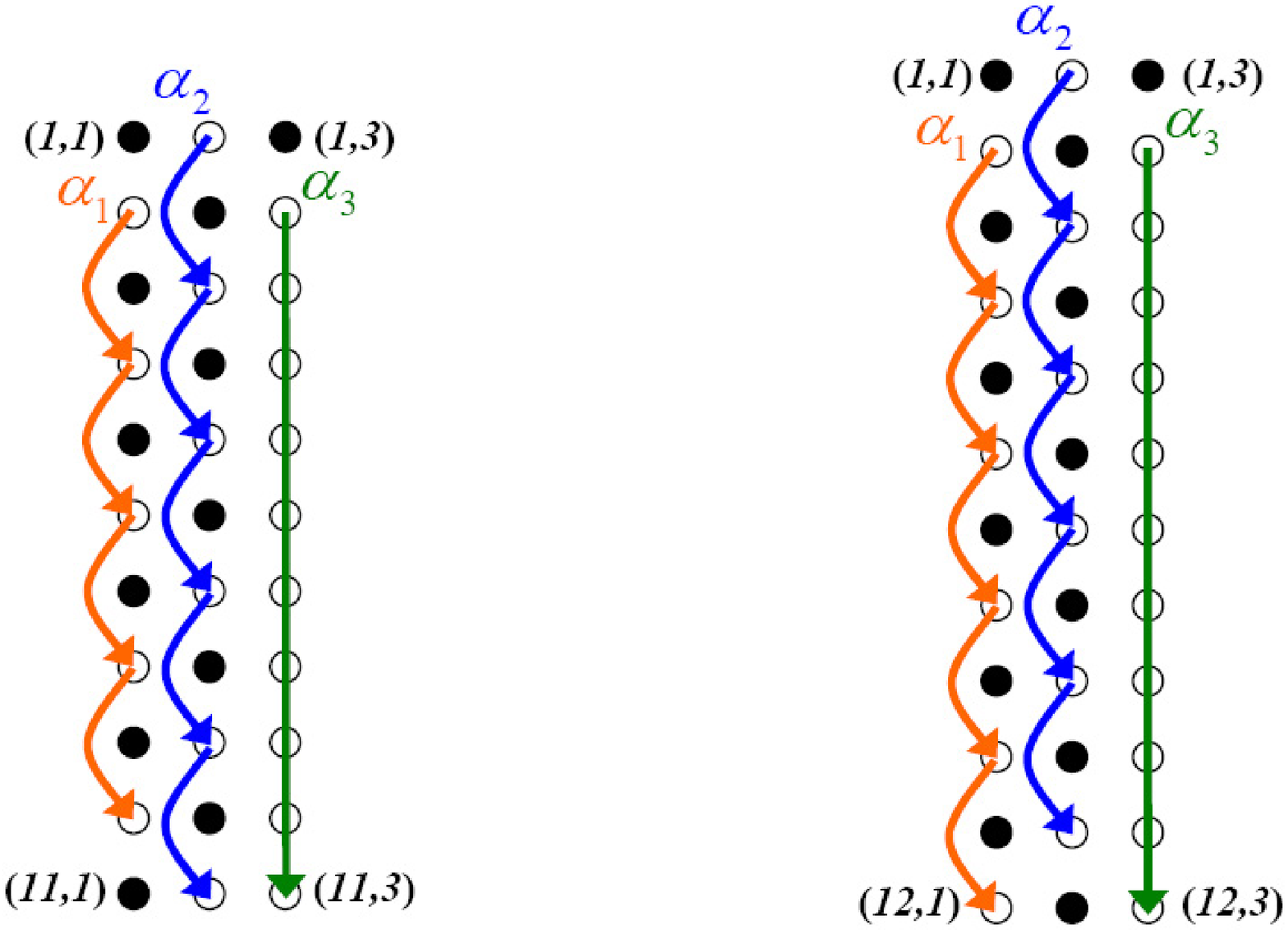}
 \begin{description}
    \item[Figure 1.] $\mbox{\rm min-seed}(C_{11} \oslash C_3,3)=12$ (left)
  and
  $\mbox{\rm min-seed}(C_{12} \oslash C_3,3)=13$ (right),
  where the target set $S$ is the set of all black vertices, and
  the convinced sequence $\alpha_1\sqcup \alpha_2\sqcup \alpha_3$ is
  illustrated by three colored directed paths.
 \end{description}
 \includegraphics[scale=0.6]{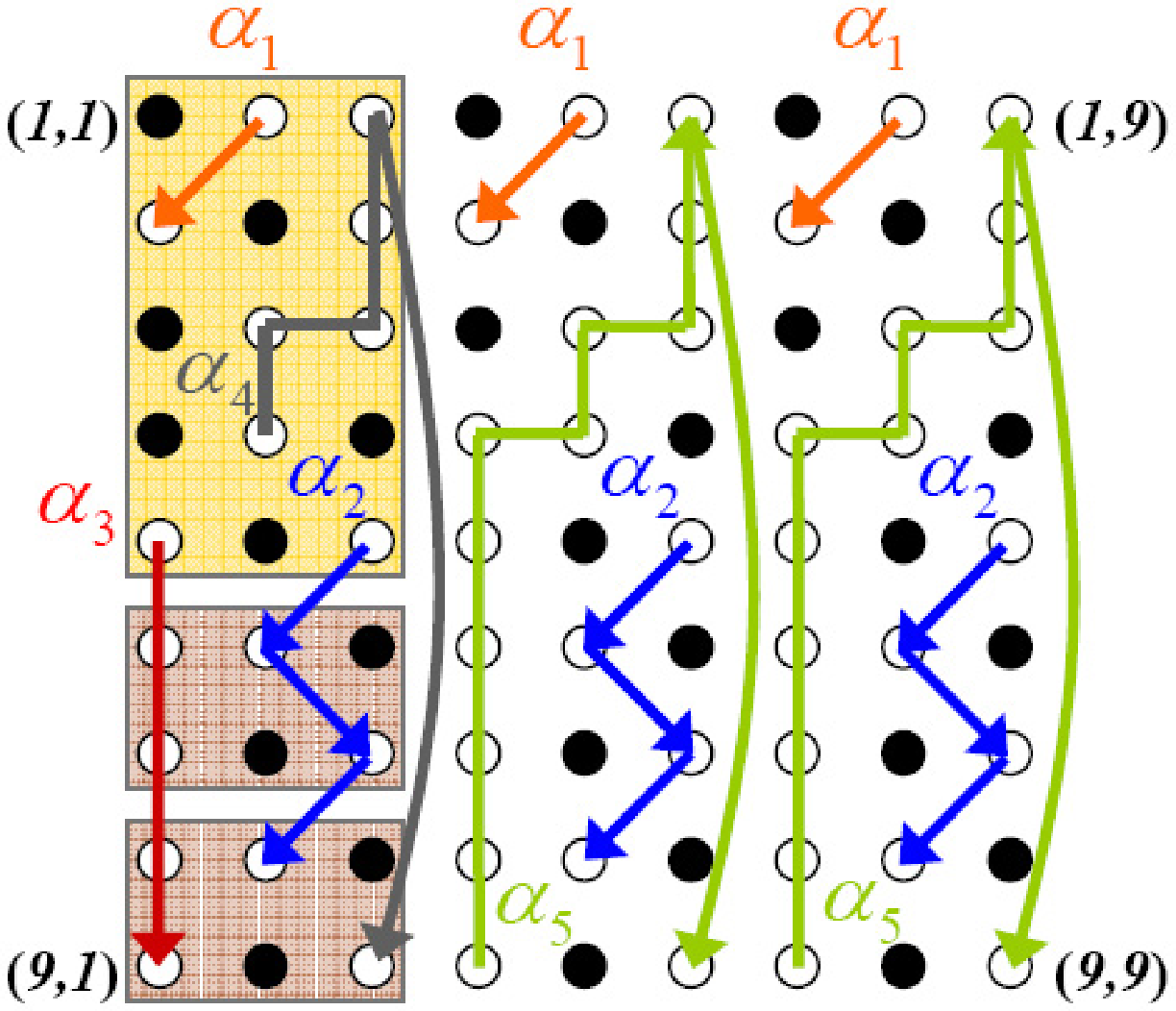}
 \begin{description}
    \item[Figure 2.]
  $\mbox{\rm min-seed}(C_9 \oslash C_{9},3)=28$.
 \end{description}
 \bigskip

 \includegraphics[scale=0.58]{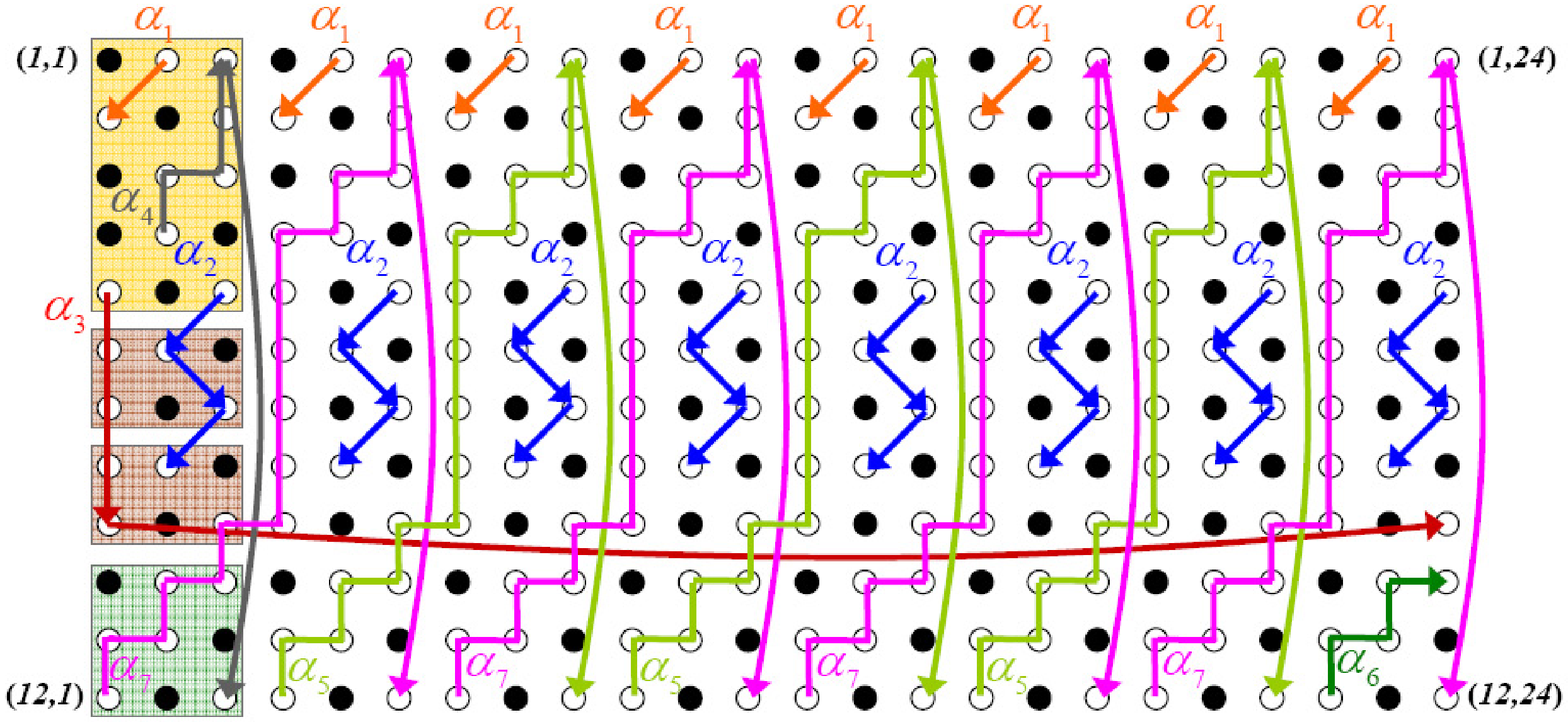}
 \begin{description}
    \item[Figure 3.]
  $\mbox{\rm min-seed}(C_{12} \oslash C_{24},3)=97$.
 \end{description}
 \bigskip

 \includegraphics[scale=0.65]{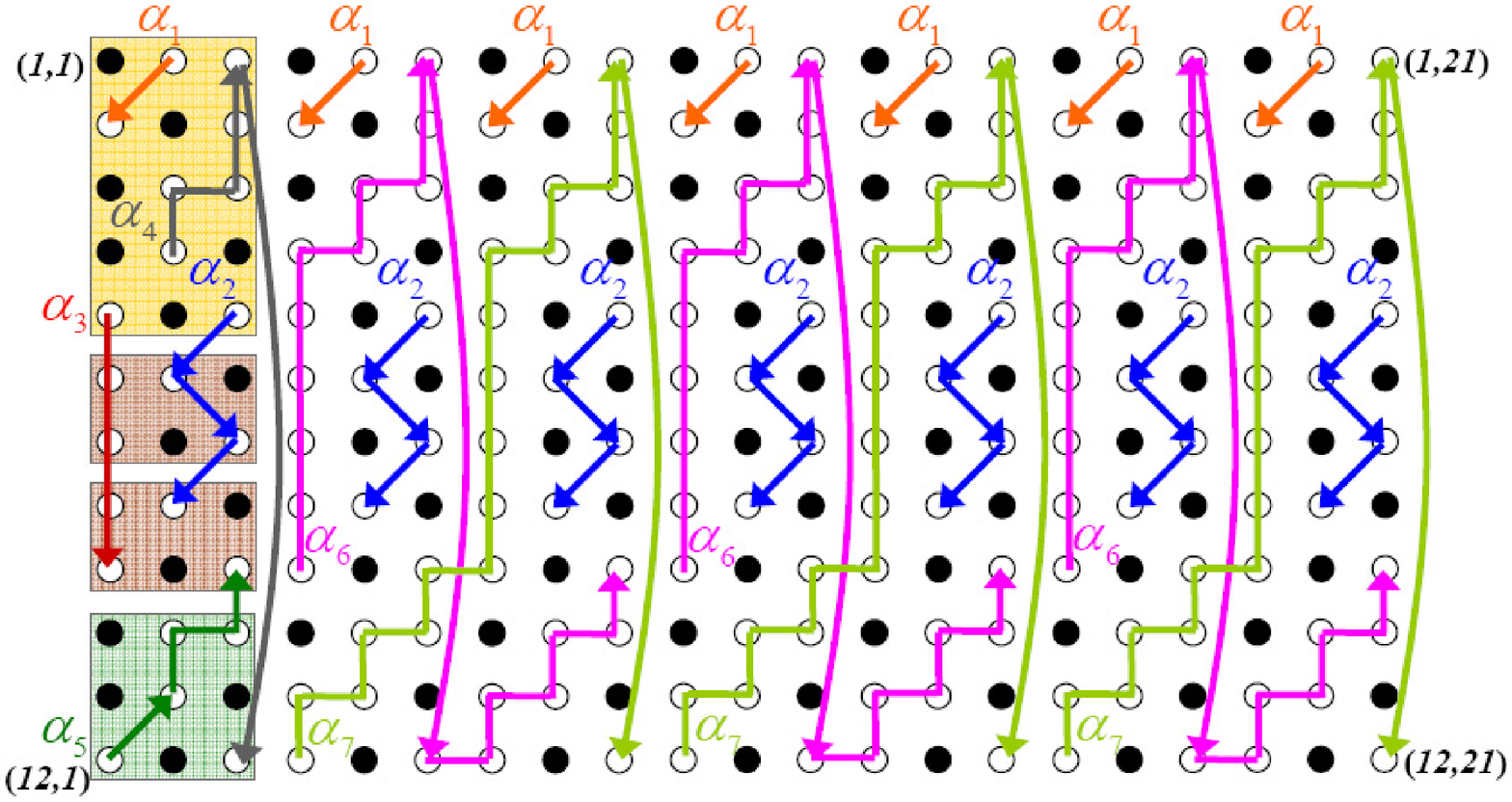}
 \begin{description}
    \item[Figure 4.]
  $\mbox{\rm min-seed}(C_{12} \oslash C_{21},3)\leq
    86$.
 \end{description}
 \bigskip

 \includegraphics[scale=0.55]{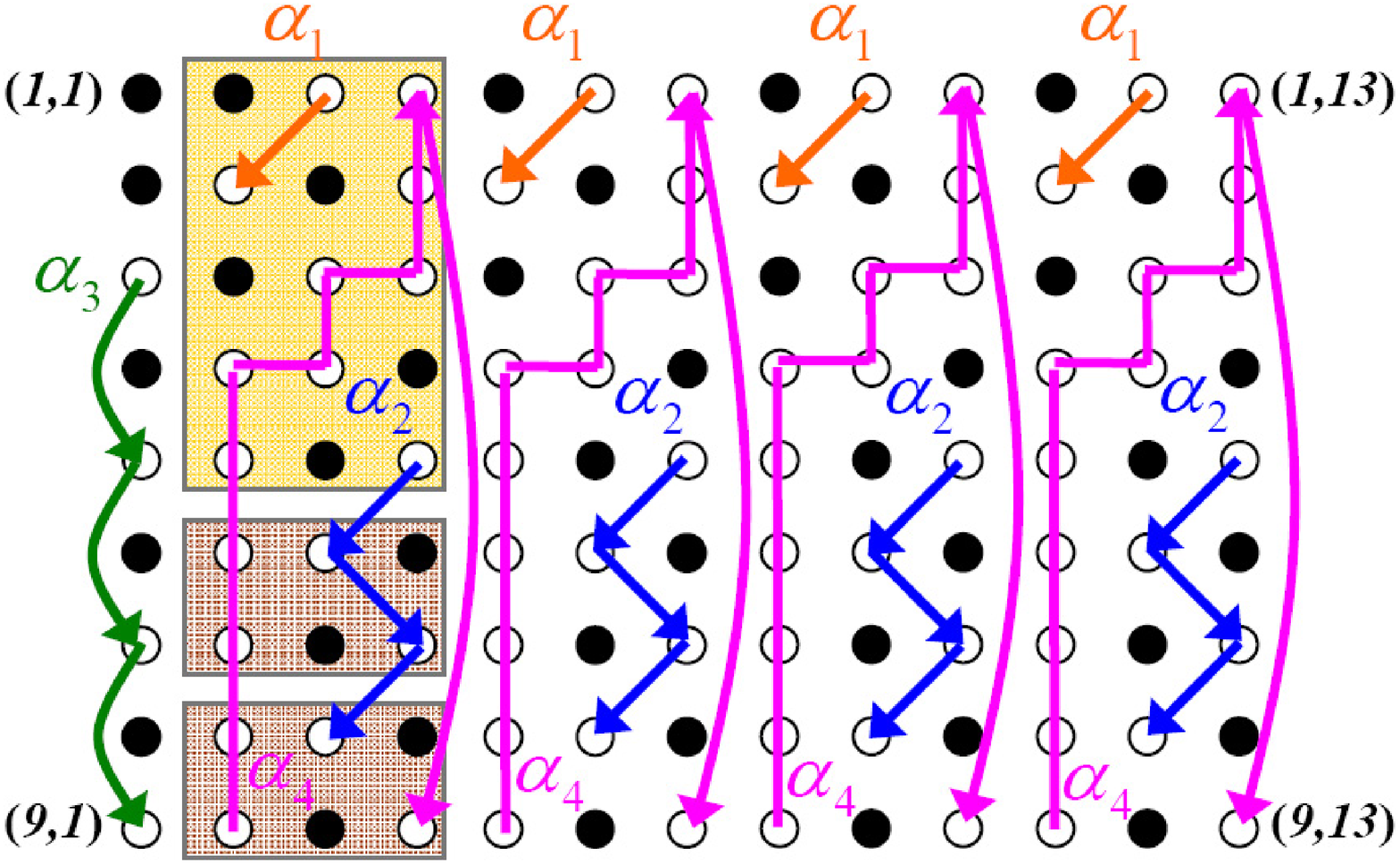}
 \begin{description}
    \item[Figure 5.]
  $\mbox{\rm min-seed}(C_{9} \oslash C_{13},3)\leq
    41$.
 \end{description}
 \bigskip

 \includegraphics[scale=0.5]{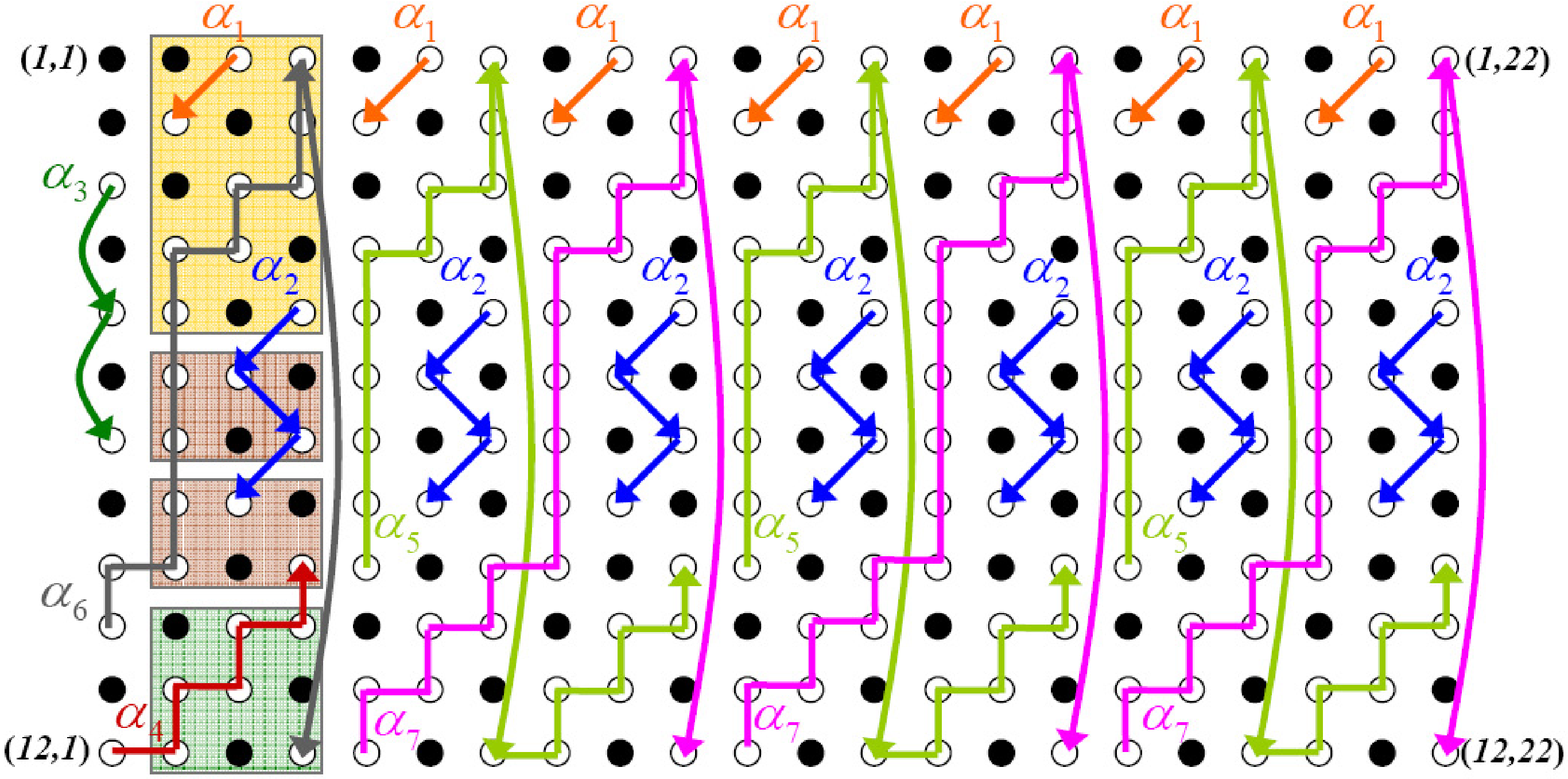}
 \begin{description}
    \item[Figure 6.]
  $\mbox{\rm min-seed}(C_{12} \oslash C_{22},3)\leq
    90$.
 \end{description}
 \bigskip

 \includegraphics[scale=0.43]{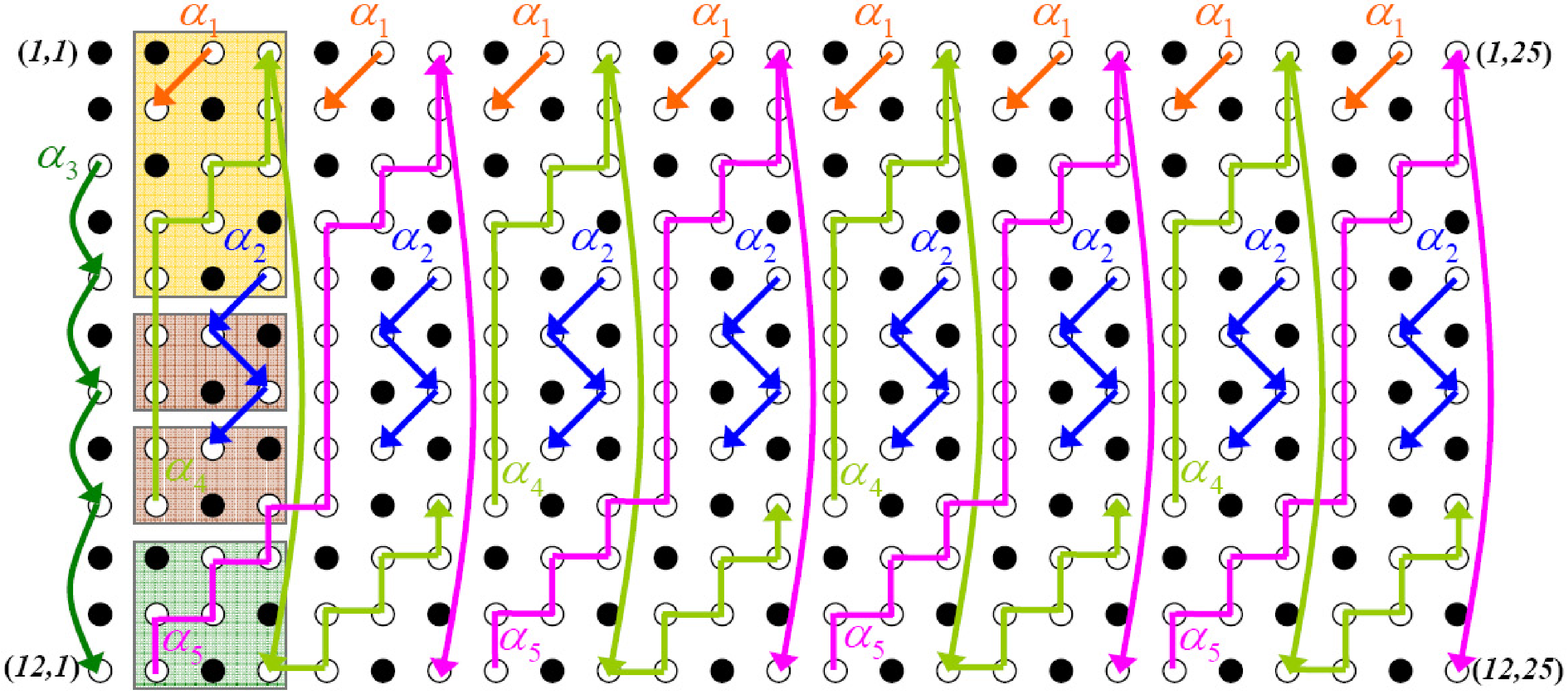}
 \begin{description}
    \item[Figure 7.]
  $\mbox{\rm min-seed}(C_{12} \oslash C_{25},3)\leq
    103$.
 \end{description}
 \bigskip

 \includegraphics[scale=0.62]{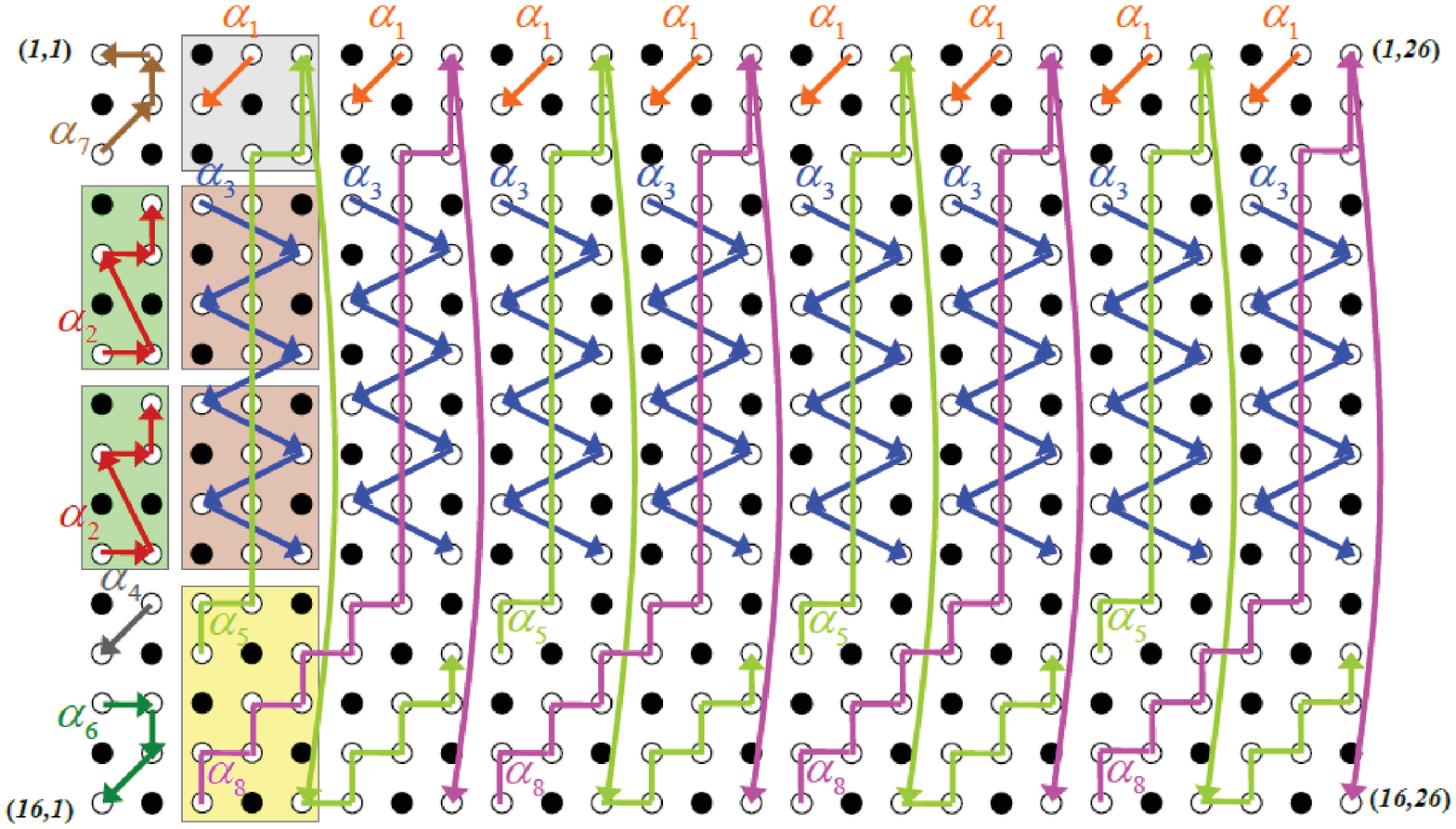}
 \begin{description}
    \item[Figure 8.]
  $\mbox{\rm min-seed}(C_{16} \oslash C_{26},3)\leq
    140$.
 \end{description}
 \bigskip

 \includegraphics[scale=0.64]{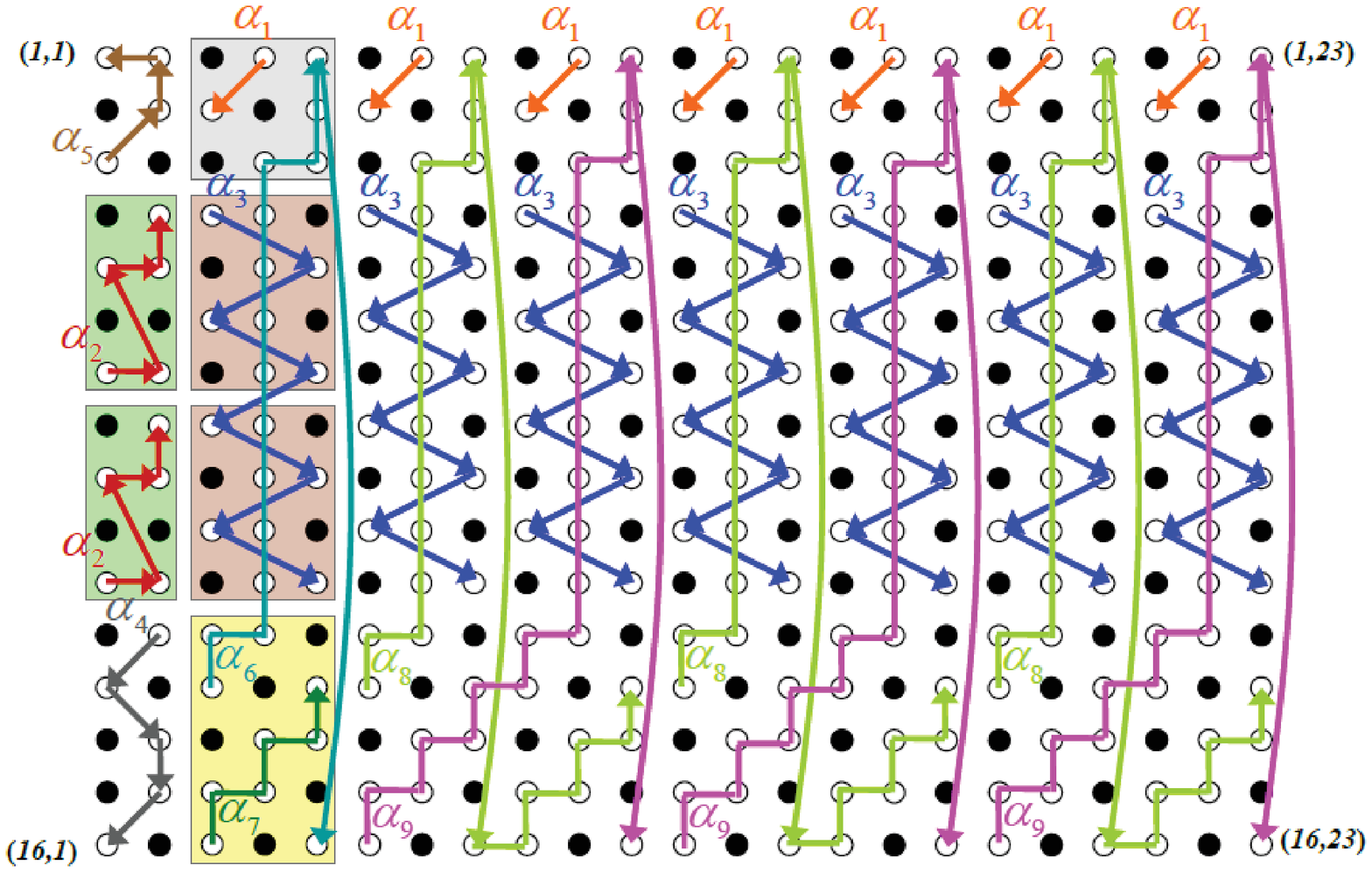}
 \begin{description}
    \item[Figure 9.]
  $\mbox{\rm min-seed}(C_{16} \oslash C_{23},3)\leq
    125$.
 \end{description}
 \bigskip

 \includegraphics[scale=0.7]{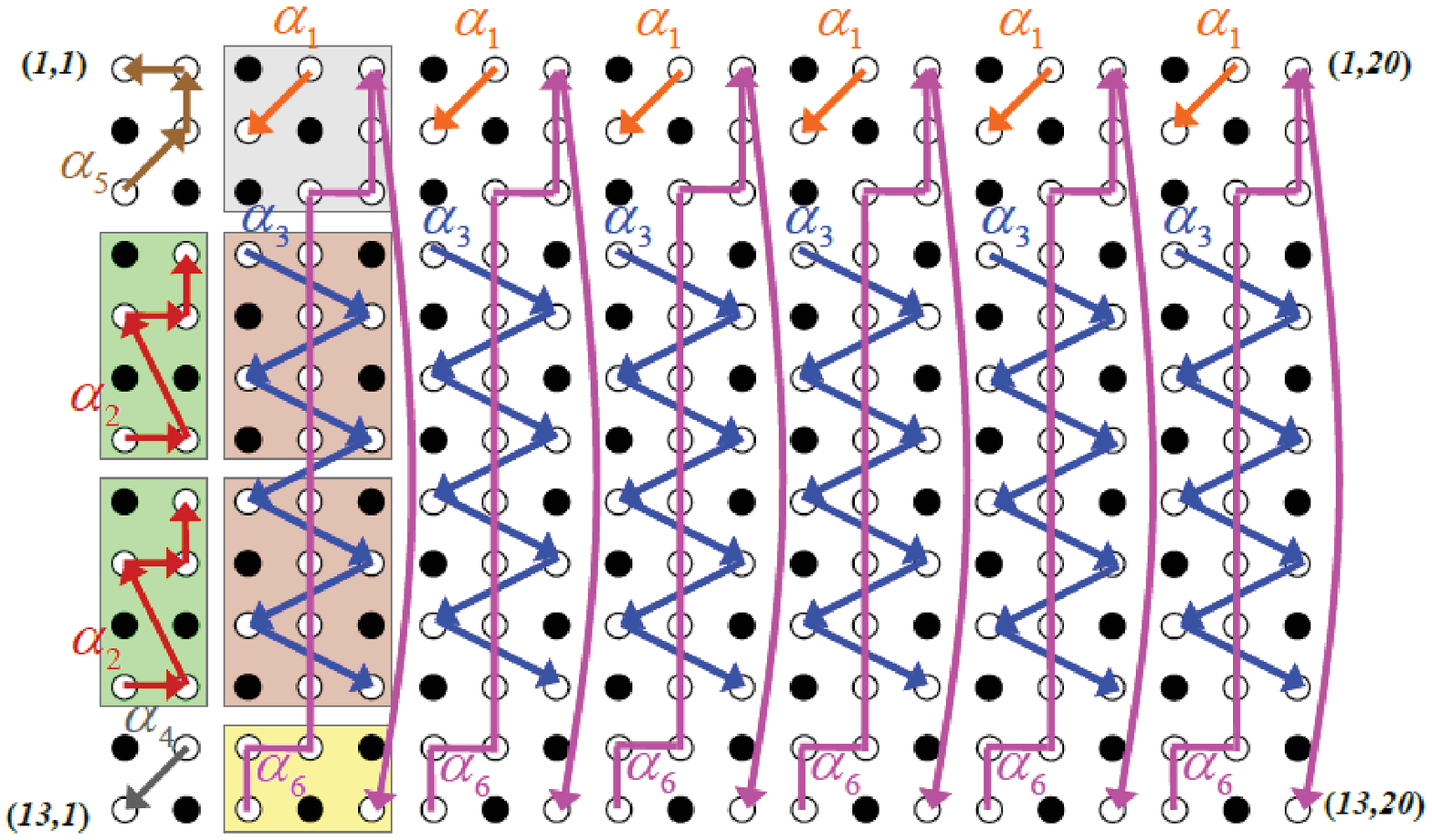}
 \begin{description}
    \item[Figure 10.]
  $\mbox{\rm min-seed}(C_{13} \oslash C_{20},3)\leq
    88$.
 \end{description}
 \bigskip

 \includegraphics[scale=0.6]{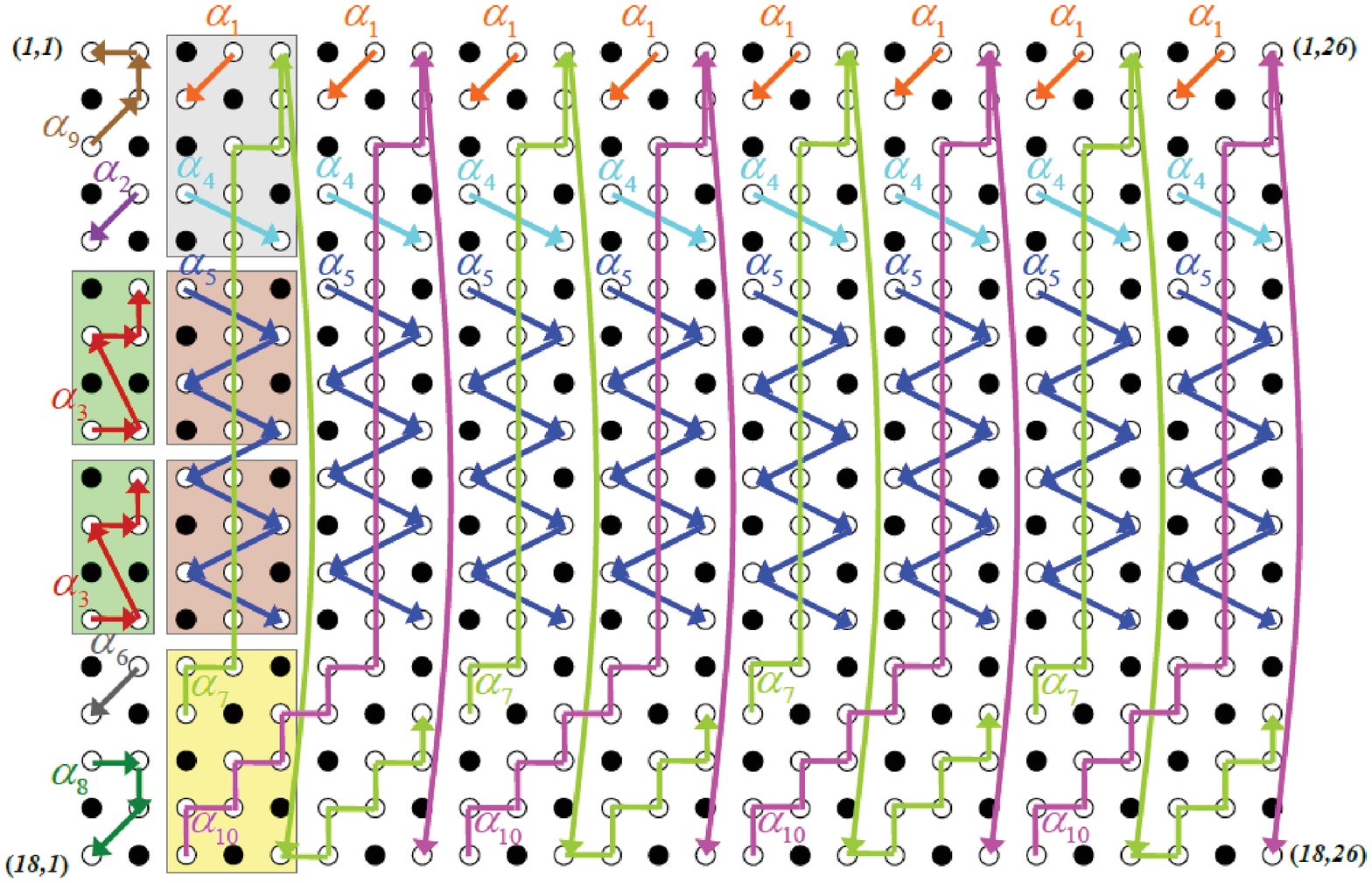}
 \begin{description}
    \item[Figure 11.]
  $\mbox{\rm min-seed}(C_{18} \oslash C_{26},3)\leq
    158$.
 \end{description}

 \includegraphics[scale=0.6]{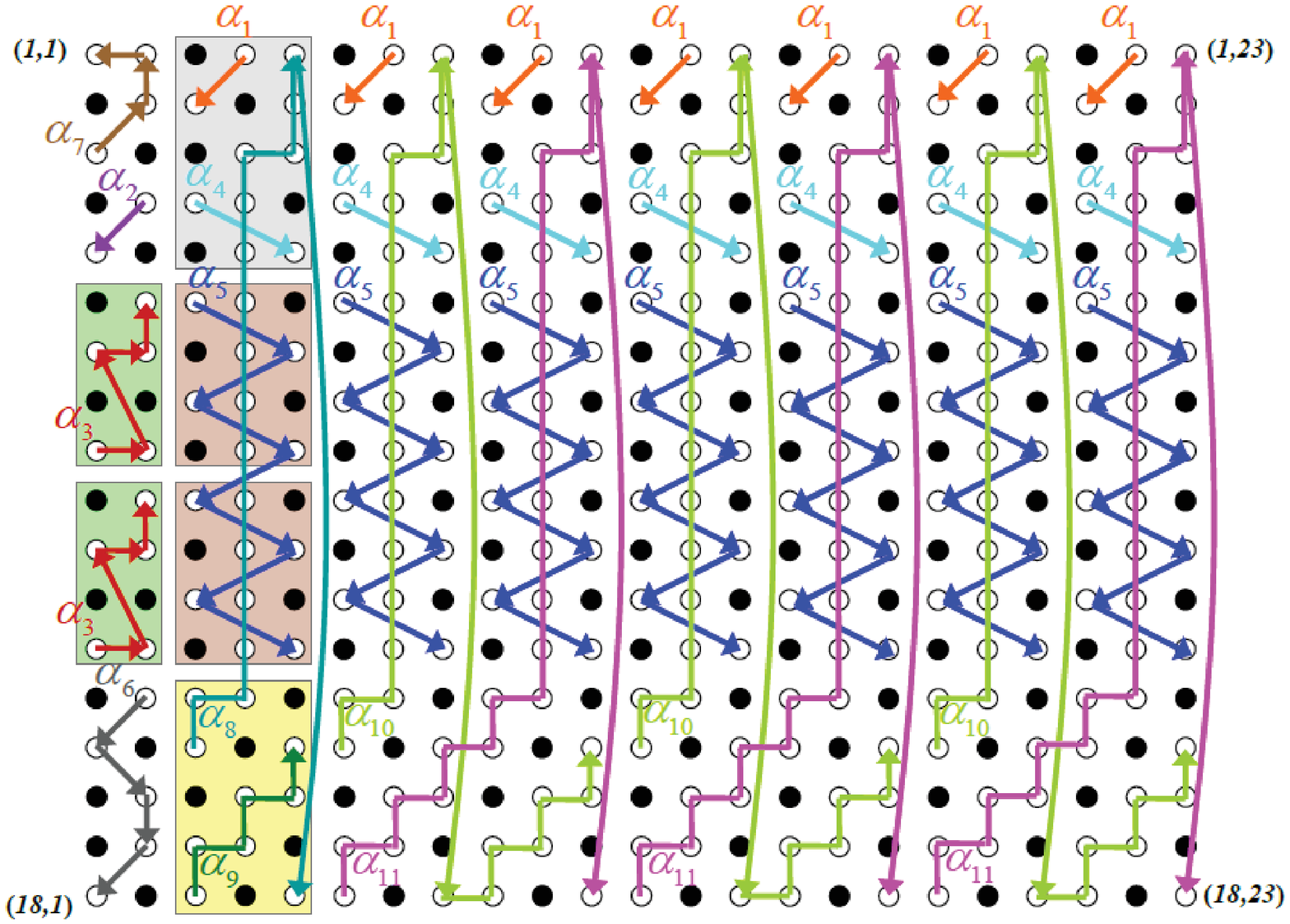}
 \begin{description}
    \item[Figure 12.]
  $\mbox{\rm min-seed}(C_{18} \oslash C_{23},3)\leq
    141$.
 \end{description}
 \bigskip

 \includegraphics[scale=0.7]{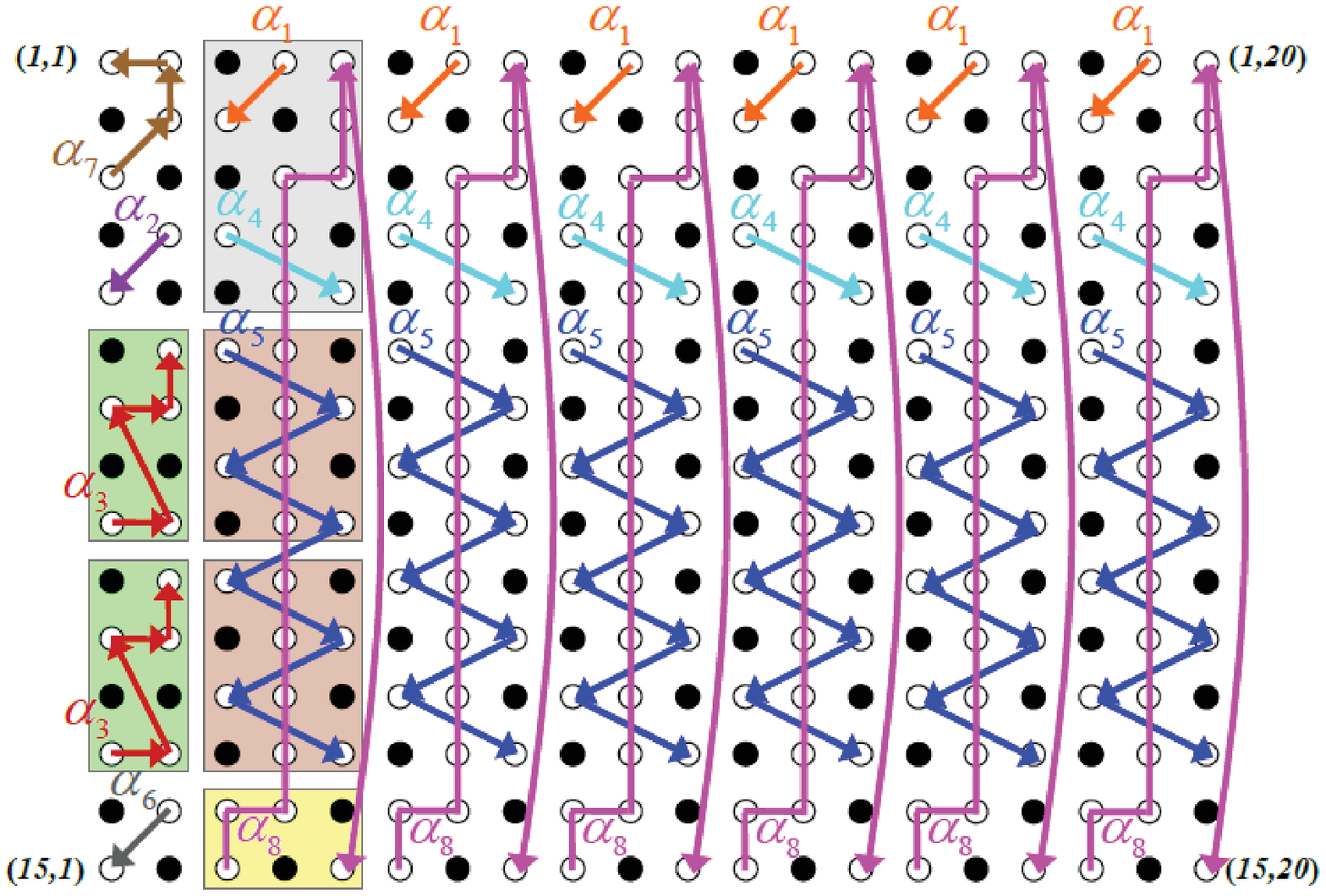}
 \begin{description}
    \item[Figure 13.]
  $\mbox{\rm min-seed}(C_{15} \oslash C_{20},3)\leq
    102$.
 \end{description}

 \includegraphics[scale=0.75]{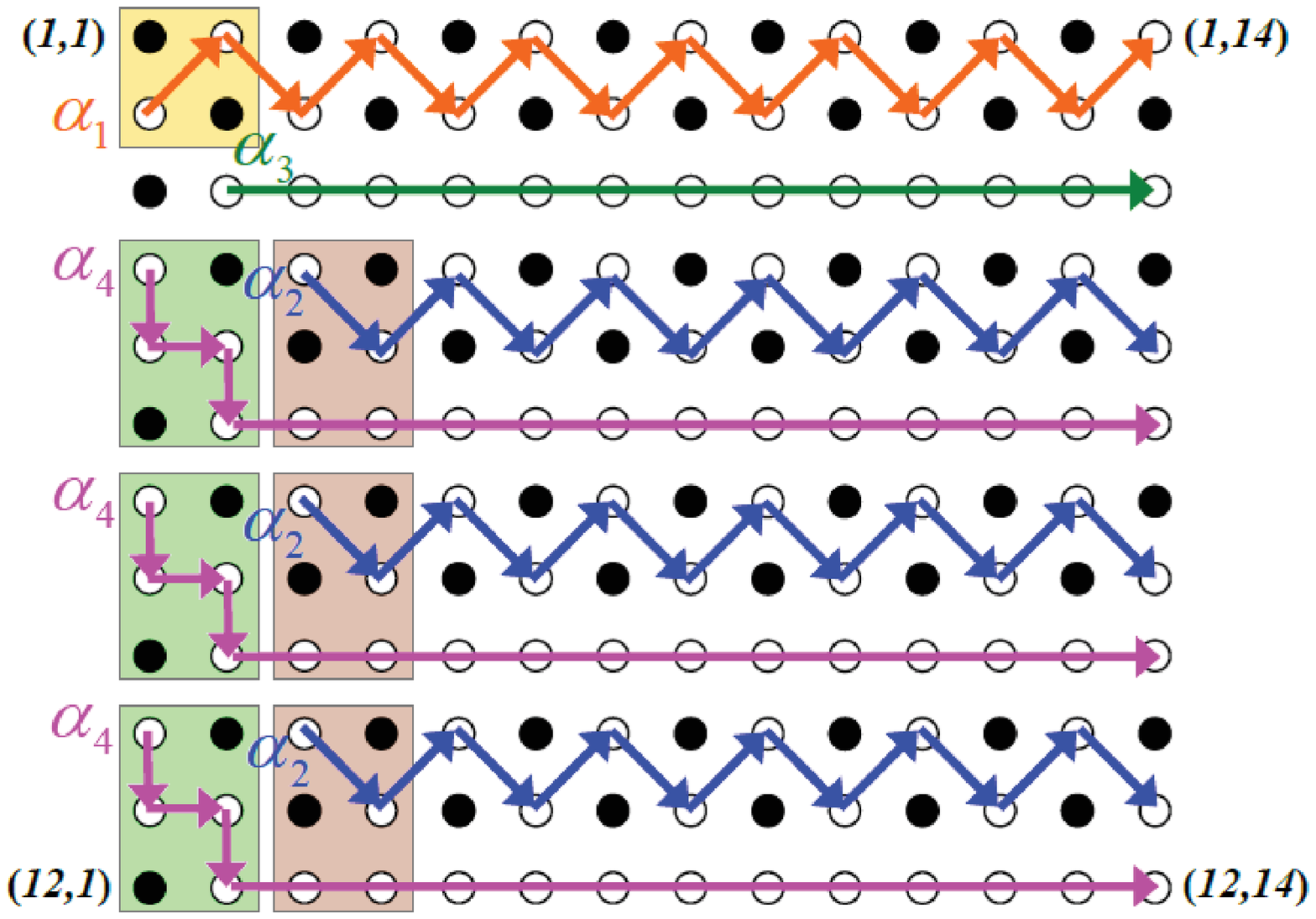}
 \begin{description}
    \item[Figure 14.]
  $\mbox{\rm min-seed}(C_{12} \oslash C_{14},3)=57$.
 \end{description}
 \includegraphics[scale=0.65]{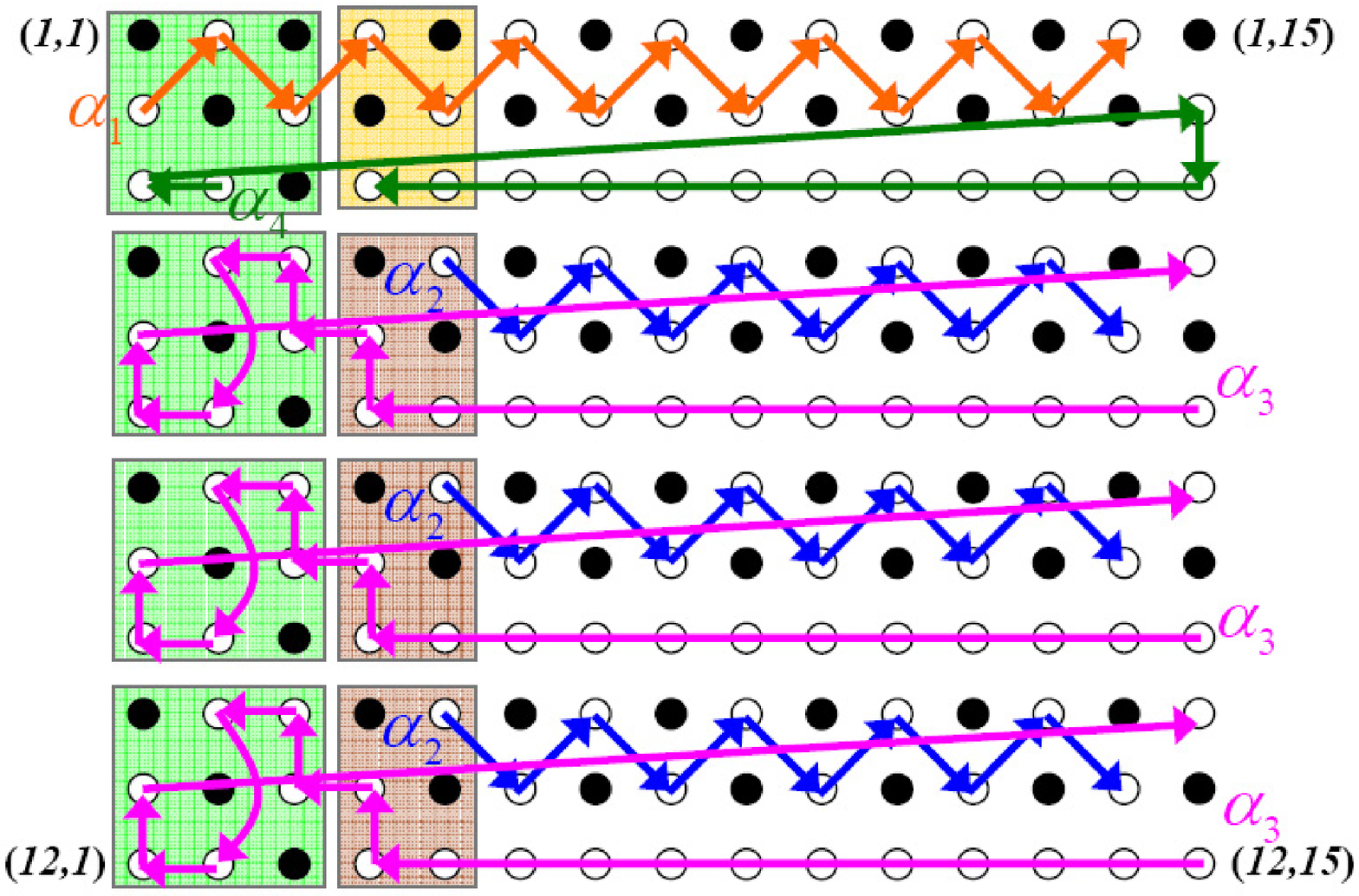}
 \begin{description}
    \item[Figure 15.]
  $\mbox{\rm min-seed}(C_{12} \oslash C_{15},3)=61$.
 \end{description}

 \end{center}
\end{document}